\documentclass[aihp]{imsart}

\RequirePackage{amsthm,amsmath,amsfonts,amssymb}

\usepackage{graphicx}
\usepackage{hyperref}
\usepackage{comment}
\usepackage{mathtools}

\usepackage{calrsfs}
\DeclareMathAlphabet{\pazocal}{OMS}{zplm}{m}{n}

\usepackage{tikz}
\usetikzlibrary{decorations.markings}

\usepackage{ytableau}

\usepackage{amssymb}
\usepackage{graphicx}
\usepackage{hyperref}
\usepackage[shortlabels]{enumitem}

\usepackage{calrsfs}
\DeclareMathAlphabet{\pazocal}{OMS}{zplm}{m}{n}

\usepackage{tikz}
\usetikzlibrary{decorations.markings}

\usepackage{ytableau}

\theoremstyle{plain}
\newtheorem{thm}{Theorem}[section]
\newtheorem{cor}[thm]{Corollary}
\newtheorem{lemma}[thm]{Lemma}
\newtheorem{prop}[thm]{Proposition}
\newtheorem{definition}[thm]{Definition}
\newtheorem{ex}[thm]{Example}

\theoremstyle{remark}
\newtheorem{remark}[thm]{Remark}

\numberwithin{equation}{section}

\newcommand{\eps}{\epsilon}

\newcommand{\Lb}{\pazocal{L}}
\newcommand{\R}{\mathbb{R}}
\newcommand{\T}{\pazocal{T}}
\newcommand{\Tb}{\mathbf{T}}
\newcommand{\X}{\pazocal{X}}

\newcommand{\Z}{\mathbb{Z}}
\newcommand{\lm}{\lambda}
\newcommand{\sst}{\Delta_n}
\newcommand{\st}{\Delta_{\infty}}
\newcommand{\ii}{\mathbf i}

\renewcommand{\d}{\mathrm{d}}

\newcommand{\pr}[1]{\mathbb{P}\left [ #1 \right]}
\newcommand{\E}[1]{\mathbb{E}\left [ #1 \right ]}
\newcommand{\ind}[1]{\mathbf{1}_{\{#1\}}}
\newcommand{\G}[1]{\Gamma \left(#1 \right)}
\newcommand{\dint}[2]{\frac{1}{(2\pi \mathbf{i})^2} \oint \limits_{C_z #1} dz \oint \limits_{C_w #2} dw\,}

\begin{document}

\begin{frontmatter}

\title{Random sorting networks: edge limit}

\begin{aug}

\author{Vadim Gorin \and Jiaming Xu}

\address[V.G.]{Departments of Statistics and Mathematics, University of California, Berkeley. vadicgor@gmail.com}

\address[J.X.]{Department of Mathematics, University of Wisconsin - Madison. jxu385@wisc.edu}

\end{aug}

\begin{keyword}[class=MSC]
\kwd{60B20}
\kwd{60B15}
\end{keyword}

\begin{keyword}
\kwd{Sorting network}
\kwd{spacing}
\kwd{antisymmetric Gaussian Unitary Ensemble}
\end{keyword}

\begin{abstract}
A sorting network is a shortest path from $1\ 2\ \dots\ n$ to $n\ \dots\ 2\ 1$ in the Cayley graph of the symmetric group $\mathfrak S_n$ spanned by adjacent transpositions. The paper computes the edge local limit of the uniformly random sorting networks as $n\to\infty$. We find the asymptotic distribution of the first occurrence of a given swap $(k,k+1)$ and identify it with the law of the smallest positive eigenvalue of a $2k\times 2k$ aGUE (an
aGUE matrix has purely imaginary Gaussian entries that are independently distributed
subject to skew-symmetry). Next, we give two different formal definitions of a spacing --- the time distance between the occurrence of a given swap $(k,k+1)$ in a uniformly random sorting network. Two definitions lead to two different expressions for the asymptotic laws expressed in terms of derivatives of Fredholm determinants.
\end{abstract}

\begin{abstract}
Un réseau de tri est un chemin le plus court de $1\ 2\ \dots\ n$ à $n\ \dots\ 2\ 1$ dans le graphe de Cayley du groupe symétrique $\mathfrak S_n$, engendré par des transpositions des éléments adjacents. Dans cet article nous calculons la limite locale au bord des réseaux de tri choisi uniformément quand $n\to\infty$. Nous trouvons la distribution asymptotique de la première occurrence  d'une transposition donnée $(k,k+1)$ et l'identifions avec la loi de la plus petite valeur propre positive d'un $2k\times 2k$ aGUE (une matrice aGUE a des entrées gaussiennes purement imaginaires qui sont distribuées indépendamment sous condition d'antisymétrie). Ensuite, nous considerons des espacements entre deux occurrences consecutives d'un échange donné (k, k + 1) pour un réseau de tri aléatoire choisi uniformément. Nous prenons deux formalisations pour un choix aléatoire d'un tel espacement. En passant à limite, ces deux définitions  conduisent à deux expressions différentes pour des lois asymptotiques exprimées en termes de dérivées des déterminants de Fredholm.
\end{abstract}

\end{frontmatter}

%\maketitle

\section{Introduction} \label{sec:intro}

\subsection{Motivation}

The main object in this article is the uniformly random sorting network which we abbreviate as RSN.  Let us start by giving basic definitions. Consider the permutation group of $n$ elements $\mathfrak{S}_{n}$. We use the one-row notation for the elements of $\mathfrak{S}_n$ representing them as sequences $(a_{1}\ a_{2}\ ...\ a_{n})$.  We let $\tau_k$ be the adjacent swap $(k,k+1)$, $1\le k\le n-1$, so that
$$
 (a_{1}\ a_{2}\ \dots \ a_{n}) \cdot \tau_k= (a_{1}\ a_{2}\ a_{k-1}\ a_{k+1}\ a_{k}\ a_{k+2}\ \dots \ a_{n}).
$$
 A sorting network of size $n$ is a shortest sequence of swaps, whose product equals the reverse permutation $(n\ n-1\ ...\ 2\ 1)$. By counting inversions in permutations, one shows that the length $N$ of such shortest sequence is $N=\binom{n}{2}$.
 Let $\Omega_n$ denote the set of all sorting networks of size $n$. Thus, elements of $\Omega_n$ are sequences $(s_1,\dots,s_N)$ such that
 $$
  \tau_{s_1} \tau_{s_2}\cdots \tau_{s_N}= (n\ n-1\ ...\ 2\ 1).
 $$
 Sorting networks can be drawn as wiring diagrams, see Figure \ref{Figure_wiring}. Whenever $s_t=k$, we say that swap $(k,k+1)$ occurs at time $t$; in the wiring diagram this corresponds to two wires intersecting at the point $(t,k+\tfrac{1}{2})$.

\begin{figure}[t]
    \begin{center}
        \includegraphics[width=0.8\linewidth]{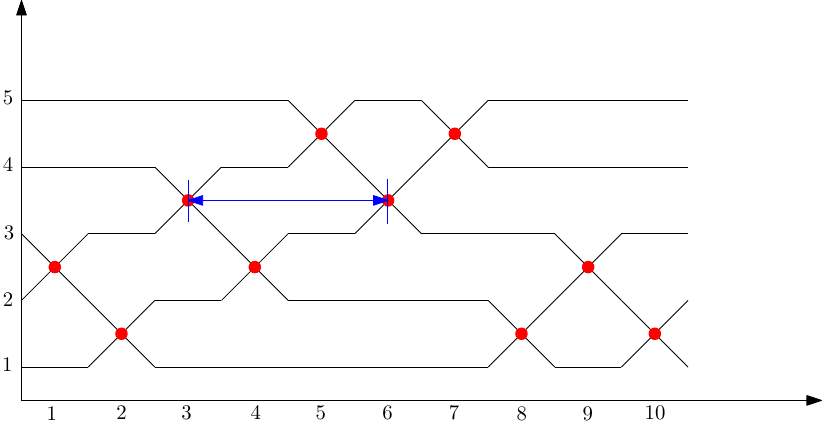}
        \caption{A sorting network $(s_1 \dots s_N)$, $N={n\choose 2}$ can be represented as a diagram of $n$ wires, with wires at heights $k$ and $k+1$ being swapped at time $i$ whenever $s_i=k$. The figure shows the wiring diagram of the sorting network $(2,1,3,2,4,3,4,1,2,1)$ with $n=5$.   The blue double arrow shows a spacing in row 3, which is a time interval between two adjacent swaps $\tau_{3}$ and it has length 3 in our example.\label{Figure_wiring}}
    \end{center}
\end{figure}
The study of sorting networks is a fruitful topic, with the first mathematical results going back to \cite{Stanley} where Stanley computed the number of elements in $\Omega_n$:
$$
  |\Omega_n|=\frac{{n\choose 2}!}{\prod_{j=1}^{n-1}(2n-1-2j)^j}.
$$
A recent point of interest is the study of uniformly random sorting networks initiated by Angel, Holroyd, Romik, and Virag in \cite{AHRV}. The probabilistic results can be split into two groups: global and local. On the global side, \cite{AHRV} proved that the density of swaps approximates the semi-circle law\footnote{There is no known direct connection to the Wigner semi-circle law of the random matrix theory and the match seems coincidental.} as $n\to\infty$; \cite{AHRV} predicted and \cite{DVi,D} proved (among other results) that individual trajectories in the wiring diagram become (random) sine curves; see also \cite{Kot,RVV} for related results.

Our paper belongs to the second group of results, \cite{AGH,Rozinov,ADHV,GR}, which studies local limits. The most basic local characteristic of the sorting network is \emph{spacing} in row $k$, which is a distance between two occurrences of the same swap $s_i=k$ in the sorting network, cf.\ Figure \ref{Figure_wiring}. \cite{Rozinov} and \cite{GR} discovered a link between the asymptotic laws of spacings as $n\to\infty$ and random matrix theory. Namely, \cite{Rozinov} matched the asymptotic law of the spacing in row $1$ with the distance between two eigenvalues of $2\times 2$ matrix of Gaussian Orthogonal Ensemble (GOE) of real symmetric random matrices.
 \cite{GR} dealt with spacings in row $\alpha n$ with $0<\alpha<1$ and showed that after proper rescaling they converge to the \emph{Gaudin--Mehta law}, which is the universal asymptotic law for spacings between eigenvalues of real symmetric matrices of very large sizes.

 Comparing the results of \cite{Rozinov} and \cite{GR}, we see that the former links the spacings in the extreme row (i.e.\ at the end-point of the edge) to $2\times 2$ real symmetric matrices, while the latter links the spacings in the bulk rows to the real symmetric matrices of infinite sizes. The observed gap in the asymptotics between $2\times 2$ and $\infty\times \infty$ matrices motivated the central question of our paper: we would like to understand, whether distributions related to $k\times k$ random matrices can be also found in the asymptotic laws for random sorting networks.

 The answer presented in this text is both Yes and No. On one side, we are so far unable to find any connections of sorting networks to eigenvalues of real symmetric $k\times k$ matrices (which would be the most direct interpolation between \cite{Rozinov} and \cite{GR}). On the other hand, by slightly adjusting our point view, we find that the law of the asymptotic spacing in row $k$ can be expressed in terms of the eigenvalues of $2k\times 2k$ random \emph{anti-symmetric} Gaussian matrices --- they are known in the literature as aGUE, since their analysis reveals connections to the tools used for the Gaussian Unitary Ensemble\footnote{However, up to multiplication by $\ii$, matrix elements of aGUE are real, rather than complex.}, see \cite[Chapter 13]{Mehta} and \cite{FN}.

\subsection{Main Results}

Here is the precise definition of the random matrix object appearing in our asymptotic results.

\begin{definition}\label{def:ague}
Let $Y$ be an $\ell \times \ell $ matrix whose entries are i.i.d.\ standard Gaussian random variables $N(0,1)$. Then
$$M_\ell=\frac{\ii}{2} \, (Y-Y^{T})$$
is called $\ell\times \ell$ anti-symmetric GUE or aGUE for being short.
\end{definition}
Note that the eigenvalues of $M_\ell$ are real and come in pairs: $\lambda$ is an eigenvalue if and only if so is $-\lambda$. When $\ell$ is odd, $M_\ell$ is necessary degenerate, i.e.\ it has a zero eigenvalue; for even $\ell$, zero is not an eigenvalue almost surely.

We also would like to deal with eigenvalues of all $M_\ell$ together.
\begin{definition}\label{def:cornerprocess}
Let $M$ be an infinite aGUE, i.e, an $\infty\times \infty$ matrix whose $l^\text{th}$ corner is a $l\times l$ aGUE. The aGUE--corners process is a point process  $\X$  on $\Z_{\ge 0}\times \R_{\ge 0}$: we put a particle at $(l,u)$ whenever $u$ is one of the $\lfloor\frac{l}{2}\rfloor$ positive eigenvalues of the top-left $l\times l$ corner of $M$.
\begin{remark}
aGUE is studied in \cite{FN}, and it turns out that its corners process $\X$ has a determinantal structure. We give the correlation kernel of $\X$ in Theorem \ref{thm:ague}.
\end{remark}
\end{definition}
\begin{definition} \label{Definition_aGUE_small}
 $\Tb(k)$ is the smallest positive eigenvalue of a $2k\times 2k$ aGUE.
\end{definition}
The determinantal structure for the distribution of eigenvalues of $M_{2k}$ (see \cite[Chapter 13]{Mehta} and \cite{FN}) leads to an expression for the distribution of  $\Tb(k)$ as a series, which can be identified with a Fredholm determinant:
$$ \pr{ \Tb(k)> t} = 1 + \sum_{l=1}^{\infty} \frac{(-1)^l}{l!} \int_0^t \dots \int_0^t \det[ K^{(k)}(u_{i},u_{j})] \, \d u_1 \cdots \d u_l,$$
where $K^{(k)}$ is a correlation kernel, expressed through the Hermite polynomials $H_i(x)$ as
\begin{equation}  \label{eq_Cor_intro} K^{(k)}(u_{1},u_{2})=e^{-\tfrac{1}{2}u_{1}^{2}-\tfrac{1}{2}u_{2}^{2}}\sum_{\ell=0}^{k-1}\frac{H_{2\ell}(u_{1})H_{2\ell}(u_{2})}{\sqrt{\pi} (2\ell)! 2^{2\ell-1} }. \end{equation}
Here we use the ``physicist's normalization'', so that $H_i(x)$, $i=0,1,\dots$, is a polynomial of degree $i$ with leading coefficient $2^i$ and such that \begin{equation} \label{eq_Hermite_ortho}
\int_{\mathbb R}H_{i}(x)H_{j}(x)e^{-x^{2}}dx=\delta_{ij}j!2^{j}\sqrt{\pi}.
\end{equation}

We can now state our first theorem.
\begin{thm}[First swapping law]\label{thm:firstswap}
 Fix $k\in \Z_{\ge 1}$, and let $\Tb_{\rm{FS},n}(k)$ denote the first occurrence time of the swap $(k,k+1)$ in a uniformly random sorting network $(s_1,\dots,s_N)\in \Omega_n$, $N={n\choose 2}$:
  $$
   \Tb_{\rm{FS},n}(k)=\min\{t\ge 1:\, s_t=k\}.
  $$
  Then we have the following convergence in law:
\begin{equation}
\label{eq_x13}
\lim_{n\rightarrow \infty} \frac{\Tb_{\rm{FS},n}(k)}{n^{\frac{3}{2}}} \, = \frac{1}{2} \Tb(k).
\end{equation}

\end{thm}
In order to connect the first swaps to the spacings, we are going to use translation invariance of the random sorting networks (cf.\ \cite[Theorem 1, (i)]{AHRV}), which is based on the following combinatorial fact:
\begin{prop}\label{prop:invariance} Let $N={n\choose 2}$. Then $(s_{1},...,s_{N})\in \Omega_{n}$ if and only if  $(s_{2},...,s_{N},n-s_{1})\in \Omega_{n}$.
\end{prop}
Applying Proposition \ref{prop:invariance} one readily proves that for all $1\le r \le r+t \le N$,
$$
  \pr{ \Tb_{\rm{FS},n}(k)> t} = \pr{ \text{there are no swaps } (k,k+1) \text{ at times }r+1,r+2,\dots,r+t},
$$
which can be used to identify the distribution of a spacing in row $k$. Before doing so, we need to specify the definition of a spacing. In fact, there are two natural definitions, which lead to distinct random variables. For the definition it is convenient to extend a sorting network to a $2N$--periodic $\mathbb Z$--indexed sequence:

\begin{definition} \label{Definition_periodic_extension}
 Given $(s_1\dots s_N)\in \Omega_n$, we extend it to a sequence $(s_t)_{t\in\mathbb Z}$ by requiring
 $s_{t+N}=n-s_t$ for all $t\in\mathbb Z$.
 We call $(s_t)_{t\in\mathbb Z}$ the \emph{periodic extension} of $(s_1\dots s_N)\in \Omega_n$.
\end{definition}
For instance $(1\, 2 \, 1)\in \Omega_3$ is extended to the infinite $\mathbb Z$--indexed sequence $(\dots 2\, 1 \, 2 \, 1 \, 2 \, \dots)$ with $1$s at odd positions and $2$s at even positions. The reason for this particular definition is contained in the following straightforward corollary of Proposition \ref{prop:invariance}:

\begin{cor}
\label{Corollary_translation_invariance}
 If $(s_1,\dots,s_N)$ is a uniformly random element of $\Omega_n$, then its periodic extension $(s_t)_{t\in\mathbb Z}$ is translation-invariant: for each $a\in\mathbb Z$ we have a distributional identity
 $$
  (s_{t+a})_{t\in\mathbb Z}\stackrel{d}{=} (s_t)_{t\in\mathbb Z}.
 $$
\end{cor}

\begin{definition}[First definition of the spacing]\label{def:first}
 Fix $n=1,2,\dots$, $k\in \{1,2,...,n-1\}$ and $a\in\mathbb Z$. Let $(s_t)_{t\in \mathbb Z}$ be the periodic extension of a uniformly random sorting network in $\Omega_n$. Set
 $$
  X=\max\{t\le a: s_t=k\}, \qquad Y=\min\{t> a: s_t=k\}.
 $$
 We define the spacing on the $k^\text{th}$ row of a sorting network of size $n$ as  $\mathrm{Sp}_{k,n}:=Y-X$.
\end{definition}
\begin{remark}
 While the definition depends on the choice of $a$, Corollary \ref{Corollary_translation_invariance} implies that the distribution of $\mathrm{Sp}_{k,n}$ is the same for all $a\in\mathbb Z$. Hence, we omit $a$ from the notation.
\end{remark}
The $n\to\infty$ limit of $\mathrm{Sp}_{k,n}$ turns out to be connected to  $\Tb(k)$ of Definition \ref{Definition_aGUE_small}.
\begin{thm}\label{thm:gap1}
Fix $k\in \Z_{\ge 1}$. We have the following convergence in distribution:
$$\lim_{n\to\infty} \frac{\mathrm{Sp}_{k,n}}{n^{\frac{3}{2}}}=\frac{1}{2} Z_{k},$$
where $Z_{k}$ is a positive random variable of density $g_{k}(x)$, $x>0$, given by
\begin{equation}
\label{eq_density_of_spacing_1}
 g_k(x)= x\, \frac{\partial^2}{\partial x^2}\bigl(\pr{ \Tb(k)> x}\bigr).
\end{equation}
\end{thm}
Here is an alternative definition of the spacing.

\begin{definition}[Second definition of the spacing]\label{def:second}
 Fix $n=1,2,\dots$, $k\in \{1,2,...,n-1\}$ and $a\in\mathbb Z$. Let $(s_t)_{t\in \mathbb Z}$ be the periodic extension of a uniformly random sorting network in $\Omega_n$. Set
 $$
  Y=\min\{t>a: s_t=k\}.
 $$
 We define the spacing $\widehat{\mathrm{Sp}}_{k,n}$ on the $k^\text{th}$ row of a sorting network of size $n$ as a random variable whose law is the distribution of $Y-a$ conditional on the event $\{s_a=k\}$.
\end{definition}
\begin{remark}
As in Definition \ref{def:first}, the choice of $a$ does not affect the distribution of $\widehat{\mathrm{Sp}}_{k,n}$.
\end{remark}

Both definitions of the spacing have their own merits. Definition \ref{def:second} is the one preferred in theoretical physics and random matrix literature, and it has been used implicitly in \cite{Rozinov}, which studies the scaling limit of the spacing of RSN on the $1^{st}$ row. However, a disadvantage of Definition \ref{def:second} is that it is hard to sample or observe $\widehat{\mathrm{Sp}}_{k,n}$, as it fails to be a random variable on the state space of all sorting networks; on the other hand, $\mathrm{Sp}_{k,n}$ is a function of a sorting network, and this is the definition used in \cite{GR} which studies the limiting behavior of RSN in the bulk. The $n\to\infty$ limit of $\widehat{\mathrm{Sp}}_{k,n}$ is still connected to  $\Tb(k)$ of Definition \ref{Definition_aGUE_small}, but in a slightly different way.
\begin{thm}\label{thm:derivative}
Fix $k\in \Z_{\ge 1}$. We have the following convergence in distribution:
$$\lim_{n\to\infty} \frac{\widehat{\mathrm{Sp}}_{k,n}}{n^{\frac{3}{2}}}=\frac{1}{2} \widehat{Z}_{k},$$
where $\widehat{Z}_{k}$ is a positive random variable of density $\widehat{g}_{k}(x)$, $x>0$, given by
\begin{equation}
\label{eq_density_of_spacing_2}
 \widehat{g}_k(x)= \frac{\sqrt{\pi}}{2}\frac{(2k-2)!!}{(2k-1)!!}\, \frac{\partial^2}{\partial x^2}\bigl(\pr{ \Tb(k)> x}\bigr).
\end{equation}
\end{thm}
\begin{remark}
 In $k=1$ case $\Tb(k)$ becomes the absolute value of  a Gaussian random variable,  $|N(0,\tfrac{1}{2})|$. Hence, $\widehat{Z}_{1}$ has density $2xe^{-x^{2}}$, $x>0$, which matches the result of \cite{Rozinov}.
\end{remark}

In addition, we present an alternative expression for $\widehat{Z}_{k}$ by identifying it with a marginal of a certain two-dimensional determinantal
point process; we refer to \cite{Bor} and references therein for general discussion of the determinantal point processes.
\begin{definition}\label{def:kernel}
Let $\X_{k}^{'}$ be a determinantal point process on $\mathbb Z_{\ge 2}\times \mathbb R_{\ge 0}$, such that with respect to the product of the counting measure on the first coordinate and the Lebesgue measure on the second coordinate, the correlation kernel is
\begin{align*}
K^{'}_{k}(x_1,u_1;\,x_2,u_2) &= \ind{u_2 < u_1,\, x_2 < x_1}\,\frac{(u_{2}-u_{1})^{x_{1}-x_{2}-1}}{(x_{1}-x_{2}-1)!}\\
\nonumber & + \dint{[0,\infty)}{[0,x_{1})} \frac{\Gamma(-w)}{\Gamma(z+1)}\frac{\Gamma(z+x_{2}+1)}{\Gamma(x_{1}-w)}\frac{\Gamma(-\frac{z+x_{2}}{2})}{\Gamma(\frac{-x_{1}+w+1}{2})}  \\
\nonumber &\quad \times \frac{z+x_{2}-2k}{z+x_{2}-2k+1}\frac{x_{1}-w-2k}{x_{1}-w-2k-1}\frac{u_{1}^{w}u_{2}^{z}}{w+z+x_{2}-x_{1}+1}.
\end{align*}
Here $u_{1},u_{2}\in \R_{>0}$ and $x_{1},x_{2}\in \Z_{\ge 2}$; when $u_1$ or $u_2$ is equal to $0$, the kernel is defined as the limit as $u_1$ or $u_2$ tends to 0. $C_{z}[0,\infty), C_{w}[0,x_{1})$ are counterclockwise--oriented contours which enclose the integers in $[0,\infty)$ and $[0,x_{1})$, respectively, and are arranged so that $C_z$ and $x_1-x_2-1-C_w$ are disjoint, as in Figure \ref{Figure-unboundedcontour}.
\end{definition}

\begin{remark}
For $m\ge 1$, there are $m$ particles in the $2m^\text{th}$ and $2m+1^\text{th}$ levels of $\X^{'}_{k}$ (i.e.\ with first coordinates $2m$ and $2m+1$, respectively), except that on the  $2k^\text{th}$ level there are only $k-1$ particles.

If we denote the $j^\text{th}$ particle(from top to bottom) on the $l^\text{th}$ level by $x^{l}_{j}$ and set $x^{l}_{j}=0$ for $j> \#$ particles in the $l^\text{th}$ level, then the particles of $\X^{'}_{k}$ interlace, i.e,
$$x^{l+1}_{j+1}\le x^{l}_{j}\le x^{l+1}_{j}.$$

While we do not prove this, we expect that compared to $\X$, $\X^{'}_{k}$ can be thought as the corner process of aGUE conditioned on the event that the smallest nonnegative eigenvalue of the $2k^\text{th}$ corner is 0, see Remark \ref{remark_conditional}.
\end{remark}
\begin{thm}\label{thm:correspondence} Let $\widehat{Z}_{k}$  be as in Theorem \ref{thm:derivative} and let $Z'_{l}$ denote the coordinate of the closest to the origin particle on the $l^\text{th}$ level in the point process $\X'_{k}$.
\begin{itemize}
\item If $k=1$, then the law of $\widehat Z_1$ coincides with that of $Z'_{3}$.
\item If $k\ge 2$, then the law of $\widehat Z_k$ coincides with that of $\min\{Z'_{2k-1}, Z'_{2k+1}\}$.
\end{itemize}
\end{thm}

\subsection{Methods}
Here is a sketch of our proof of Theorem \ref{thm:firstswap}:
\begin{itemize}
\item The Edelman-Greene bijection \cite{EG} maps the problem to the study of the asymptotic distribution of the entries of uniformly random standard Young tableau of staircase shape, see Section \ref{Section_EG}.
\item We replace standard Young tableau by a continuous version, which we call Poissonized Young tableau. In Section \ref{Section_SYT_to_PYT} we show that the error of this replacement is negligible in our limit regime.
\item We use the formula of \cite{GR} for the correlation kernel of the determinantal point process, describing the entries of a Poissonized Young tableau.
\item Asymptotic analysis of this formula leads in Theorems \ref{thm:def1} and \ref{thm:def2} to the limiting process described in terms of double contour integral formulas for its correlation kernel.

\item Expanding the integrals in residues, performing resummations, and using the Gibbs (conditional uniformity) properties of the point processes under consideration, we reveal in Section \ref{subsec:check} a match to the distribution of  eigenvalues of aGUE matrices.
\end{itemize}
We remark that (in contrast to the results in the bulk, i.e., for $k$ of order $\alpha n$, $0<\alpha<1$, as in \cite{GR,ADHV}) we do not claim any results for the joint asymptotic distribution of several swaps; Theorem \ref{thm:firstswap} only deals with one-dimensional marginal distribution. A technical reason is that the Edelman-Green bijection does not have any straightforward continuous version acting on the eigenvalues of aGUE of finite sizes: it seems that one needs to deal simultaneously with aGUE of all sizes, which is not a particularly transparent operation. In future, it would be interesting to find a way to overcome this difficulty.

\smallskip

Theorems \ref{thm:gap1} and \ref{thm:derivative} are proven in Section \ref{sec:correspondence} as corollaries of Theorem \ref{thm:firstswap}. The central idea is to develop discrete versions of  \eqref{eq_density_of_spacing_1}, \eqref{eq_density_of_spacing_2} valid for random sorting networks of finite sizes. In fact, these discrete statements are valid for any periodic point processes and, thus, can be of independent interest, see Proposition \ref{prop:circle}.

Finally, Theorem \ref{thm:correspondence} is proven in Section \ref{sec:correspondence} by repeating the arguments of Theorem \ref{thm:firstswap} for the sorting networks conditional on the event $\{s_1=k\}$. This requires asymptotic analysis of entries of standard Young tableau of a slightly different shape, which is still possible by our methods. Thus, we arrive at the identities of  Theorem \ref{thm:correspondence} in an indirect way, by showing that the two sides of the identity are $n\to\infty$ limits of the same discrete distributions.

\subsection*{Acknowledgements}  We thank the anonymous referee for helpful comments. The work of V.G.\ was partially supported by NSF Grants DMS-1664619, DMS-1949820, DMS-2246449, and by the Office of the Vice Chancellor for Research and Graduate Education at the University of Wisconsin--Madison with funding from the Wisconsin Alumni Research Foundation.

\section{Preliminaries}\label{sec:pre}

One of the key tools in studying sorting networks is the Edelman-Green bijection (see Section \ref{Section_EG} for the details), which maps them to Standard Young Tableaux (SYT). In this and the next section we develop asymptotic results for SYT, which will be ultimately used in Section \ref{sec:correspondence} to prove the theorems about the random sorting networks announced in the introduction.

In addition, in the last subsection of this section we recall the properties of the eigenvalues of aGUE, which will be useful later on.

\subsection{Young diagrams and Young tableaux}\label{subsec:youngtableaux}
A partition is a sequence of non-negative integers $\lambda_{1}\ge \lambda_{2}\ge...\ge 0$ such that $|\lambda|:=\sum_{i=1}^{\infty}\lambda_{i}< \infty$. The length of $\lambda$, denoted by $l(\lambda)$, is the number of strictly positive $\lambda_{i}s$.

We identify a partition $\lambda$ with a collection of $N=|\lambda|$ boxes, such that there are $\lambda_{i}$ boxes in row $i$. We use coordinate $(i,j)$ to denote the j-th box at row i. In other words, the coordinates of boxes are $\{(i,j)\mid 1\le j\le \lambda_{i},\, i=1,2,\dots\}$. Such a combinatorial object is named a \emph{Young diagram}, still denoted as $\lambda$. In particular, we say the Young diagram is of staircase shape and write $\lambda=\Delta_{n}$ for some $n\in \Z_{\ge 2}$, if $\lambda_{i}=n-i$ for $i=1,2,\dots,n$. We also use the diagram $\lambda=\Delta_{n}\setminus(n-k,k)$ for $1\le k \le n-1$, which has $\lambda_i=n-i$ for $i\ne n-k$ and $\lambda_{n-k}=k-1$.

A \emph{standard Young tableau} (SYT) of shape $\lambda$ is an insertion of numbers 1,2,...,$|\lambda|$ into boxes of $\lambda$ without repetitions, such that the numbers in each row and column strictly increase from left to right and from top to bottom. The numbers within a SYT are its \emph{entries}. Fixing the Young diagram $\lambda$, we can get a uniformly random SYT of shape $\lambda$, denoted as $\Tb_{\lambda}$ by considering a uniform measure on the space of all SYT of shape $\lambda$. See Figure \ref{Figure-tableaux} for an example with $\lambda=\Delta_{6}$.

A \emph{Poissonized Young tableau} (PYT) of shape $\lambda$ is an insertion of $|\lambda|$ real numbers in $[0,1]$ into boxes of $\lambda$, such that the numbers strictly increase along each row and column. We use $\mathrm{PYT}(\lambda)$ to denote the space of all PYT of shape $\lambda$. Note that a given PYT of shape $\lambda$ can be transformed canonically to a SYT of the same shape, by replacing the k-th smallest entry with k. In the opposite direction, we can get a uniformly random PYT of shape $\lambda$, denoted as $\T_{\lambda}$ by the following steps: first generate a uniformly random SYT of shape $\lambda$, then independently sample $|\lambda|$ i.i.d.\ random variables uniform in [0,1], and replace the entry $k$ with the k-th smallest random variable.

\begin{figure}[htpb]
    \begin{center}        \includegraphics[scale=0.55]{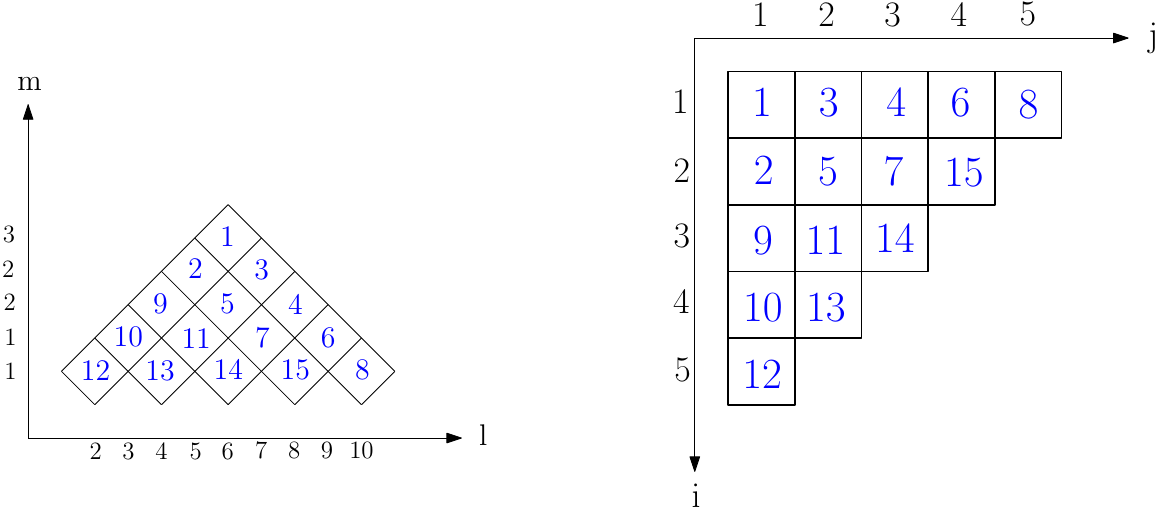}
        \caption{\small{An example of staircase shaped standard Young tableau, with $n=6$ in two different coordinate systems: the right one is defined in Section  \ref{subsec:youngtableaux} and left one is defined in Section \ref{subsec:rotation}.}\label{Figure-tableaux}}
    \end{center}
\end{figure}

\subsection{Rotation of Young Tableaux}\label{subsec:rotation}

Let $T_{\Delta_{n}}$ denote a standard Young tableau of shape $\Delta_{n}$. We make a change of coordinates $(i,j)\to(l,m)$ for the entries in $T_{\Delta_{n}}$ by letting
\begin{align*}
l=n-i+j,\qquad
m=
\begin{cases}
\frac{n-i-j}{2}+1, & n-i-j\;\; \text{is even}, \\
\frac{n-i-j+1}{2}, & n-i-j\;\; \text{is odd}.
\end{cases}
\end{align*}
We let $T_{\Delta_{n}}(l,m)$ denote the entry of $T_{\Delta_{n}}$ on the $i^\text{th}$ row and the $j^\text{th}$ column, where $i$ and $j$ are reconstructed from $l$ and $m$ by inverting the above formulas. In more detail, we have
$$
 i=\begin{cases}  n-\frac{l}{2}-m+1, & l \text{ is even},\\   n-\frac{l}{2}-m+\frac{1}{2}, & l \text{ is odd}; \end{cases} \qquad j=\begin{cases}  \frac{l}{2}-m+1, & l \text{ is even},\\   \frac{l}{2}-m+\frac{1}{2}, & l \text{ is odd}. \end{cases}
$$
 The allowed values for $(i,j)$ are: $i\ge 1$, $j\ge 1$, $i+j\le n$. The allowed values for $(l,m)$ are $l=2,3,4,\dots,2n-2$ and $m=1,2,\dots,\lfloor \frac{\min(l,2n-l)}{2}\rfloor$. The transformation $(i,j)\to (l,m)$ is essentially a clockwise rotation by  45 degrees, as can be seen in Figure \ref{Figure-tableaux}.

\bigskip

Similarly, for standard Young tableau $T_{\Delta_{n}\setminus(n-k,k)}$ we make the same change of coordinates as above and deal with $T_{\Delta_{n}\setminus(n-k,k)}(l,m)$. Formally, the entry at $(l,m)=(2k,1)$ does not exist, because it corresponds to the removed box at $(n-k,k)$, but we can also think about this entry as being $N={n\choose 2}$; in this way  $T_{\Delta_{n}\setminus(n-k,k)}$ becomes $T_{\Delta_{n}}$ with the additional restriction $T_{\Delta_{n}}(2k,1)=N$.

For Poissonized Young tableau (PYT) $\T_{\lambda}$ of shapes $\lambda=\Delta_{n}$ or $\lambda=\Delta_{n}\setminus (n-k,k)$, we define the change of coordinate in exactly the same way and use $(m,l)$ coordinate system.

\subsection{Point processes associated with PYT}\label{subsec:embd}
Let $\T_{\lambda}$ be a $\mathrm{PYT}$ of shape $\lambda$. As in Section \ref{subsec:rotation}, we will focus on the case $\lambda=\Delta_{n}$ or $\lambda=\Delta_{n}\setminus (n-k,k)$, for which we induce a point process on $\Z_{\ge 0}\times \R_{\ge 0}$ from its entries.

\begin{definition}\label{def:tableau}
The projection of a $\mathrm{PYT}$ $\T$
with shape ${\Delta_{n}}$ or $\Delta_{n}\setminus (n-k,k)$ is a point configuration on $\Z_{\ge 0}\times [0,1]$, such that there's a particle at $(l,u)$, if for some $l\in \Z_{\ge 2}$ and $m\in \Z_{\ge 1}$,
$$u=1-\T(l,m).$$
\end{definition}
By definition, $u=1-\T(l,m)$ is the $m^\text{th}$ lowest particle on level $\{l\}\times \R_{\ge 0}$ (except for level $l=2k$ in $\T_{\Delta_{n}\setminus (n-k,k)}$, where $\T(2k,1)$ is missing), and the particles on neighboring levels interlace by the setting of PYT, see Figure \ref{Figure-Projection}.

\begin{figure}[htpb]
    \begin{center}
        \includegraphics[width=0.65\linewidth]{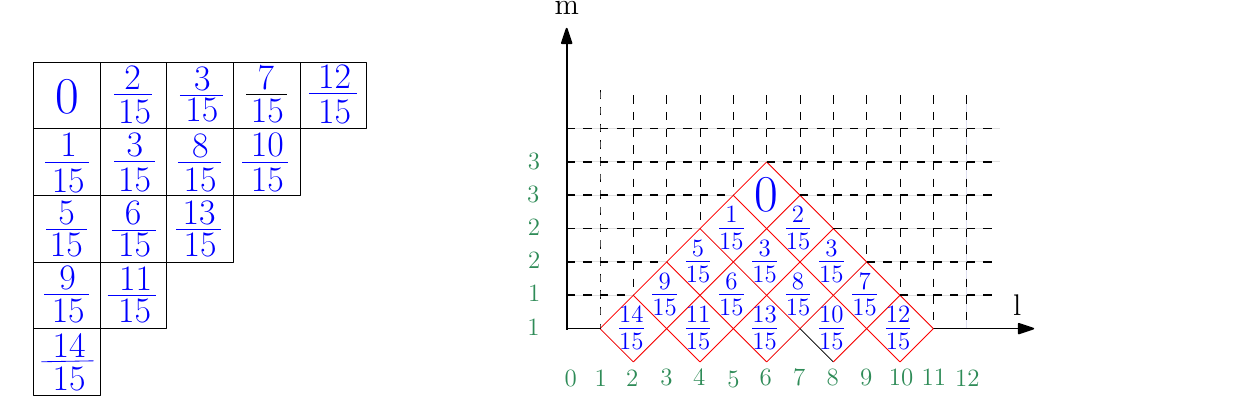}
        \includegraphics[width=0.34\linewidth]{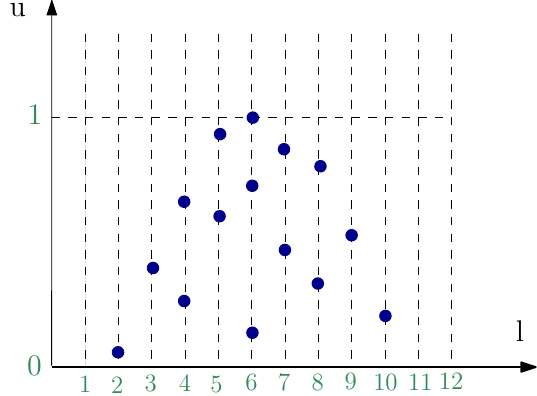}
        \caption{\small{Producing a point configuration from a $\mathrm{PYT}$ $T$ of shape $\Delta_{6}$ with entries given in the figure as in Section \ref{subsec:rotation} and \ref{subsec:embd}: first rotate $\Delta_{6}$ clockwise by 45 degrees, then project its entries into a configuration of interlacing particles. }\label{Figure-Projection}}
    \end{center}
\end{figure}

If we take $\T_{\lambda}$ to be uniformly random, then the projection of $\T_{\lambda}$ becomes a point process on $\Z_{\ge 0}\times [0,1]$. Note that this random point configuration is simple almost surely. We also would like to rescale the $u$--coordinate.
\begin{definition}\label{def:pointprocess} The point process $\X^{'}_{n}$ on $\Z_{\ge 0}\times \R_{\ge 0}$ is obtained by taking
a uniformly random $\mathrm{PYT}$ $\T_{\Delta_{n}}$ , projecting its entries on $\Z_{\ge 0}\times [0,1]$, and then rescaling the second coordinate of each particle by the map $(l,u)\mapsto (l,n^{\frac{1}{2}}\cdot u)$. Similarly, $\X^{'}_{k,n}$ is obtained from uniformly random PYT $\T_{\Delta_{n}\setminus (n-k,k)}$ by the same procedure.
\end{definition}

We  study the asymptotic behavior of $\X_{n}^{'}$ and $\X^{'}_{k,n}$ in Section \ref{subsec:limit}. Our analysis uses the technique of determinantal point processes with correlation kernels given by double contour integrals, which we present next.

\subsection{Determinantal Point Process}\label{subsec:dpp}
Here's a brief review of some standard definitions, for a thorough
introduction see \cite{Bor, DV}. Let $S$ be a locally compact Polish space, say $\Z$ or $\R^{n}$. We consider the space of discrete subsets in $S$, which consists of countable subsets $X$ of $S$ without accumulation points, and identify each such point configuration with the measure $\sum_{x\in X}\delta_{x}$. The topology of this space is given as the weak topology of Radon measures on $S$. We further use Borel $\sigma-$algebra corresponding to this topology.

A point process is a random discrete subset of $S$, and we assume that it is simple a.s, i.e, all the points have distinct positions. A \emph{determinantal point process} $\X$ on $S$ is a point process with correlation kernel $K:S\times S\rightarrow \R$, and a reference measure $\mu$ on $S$, such that for any $f: S^{k}\rightarrow \R$ with compact support,
\begin{equation} \label{eqn:dpp}
\mathbb{E} \,\Bigg [\sum_{\substack{x_1,\ldots, x_k \in \X \\ \,x_1, \ldots, x_k \; \text{are distinct}}} f(x_1, \ldots, x_k)\Bigg] =
\int \limits_{S^k} \det \left [ K(x_i,x_j) \right] f(x_1,\ldots,x_k) \,\mu^{\otimes k}(dx_1,\ldots,dx_k).
\end{equation}
Expectations of the form given by the l.h.s.~of (\ref{eqn:dpp}) determine the law of $\X$ under mild
conditions on $K$, see \cite{Lenard}. This will always be the case in this text as the correlation kernels we
consider will be continuous.

The determinantal point processes appearing in this paper live in space $S=\Z\times [0,1]$ or $S=\Z\times \R_{\ge 0}$, with reference measure being the product of counting measure and Lebesgue measure, denoted by $\#_{\Z}\times \Lb([0,1])$ or $\#_{\Z}\times \Lb(\R_{\ge 0})$ resp. We use the following lemma, whose proof we omit, see \cite{DV}, \cite{Lenard}.
\begin{lemma}\label{lemma:conv}
Let $Y_{n}$ be a determinantal point process on $\Z\times \mathbb R_{>0}$ with reference measure $\#_{\Z}\times \Lb(\mathbb R_{>0})$ and correlation kernel $K_{n}$. Then the sequence $Y_{n}$ converges weakly as $n\to\infty$ to a determinantal point process $\X$ with correlation kernel $K$ on $\Z\times \R_{\ge 0}$ (with reference measure $\#_{\Z}\times \R_{\ge 0}$) and with almost surely no particles at $(i,0)$, $i\in\mathbb Z$, if

  $K_{n}\rightarrow K$ uniformly on compact subsets of $\Z\times \R_{\ge 0}$.
\end{lemma}

\subsection{Uniformly Random Projection}\label{subsec:jumps}

When $\mathrm{PYT}$ $\T_{\Delta_{n}}$ and $\T_{\Delta_{n}\setminus (n-k,k)}$ are uniformly random, their projections are point processes on $\Z_{\ge 0}\times [0,1]$, whose correlation functions were computed in \cite{GR}. We restate the result there for our case in the following theorem, which is a stepping stone of the proofs of our main results.

\begin{thm}\label{thm:jumps} The processes $\X^{'}_{n}$ and $\X^{'}_{k,n}$ of Definition \ref{def:pointprocess} are determinantal point process on $\{2,3,\dots,2n-2\} \times \mathbb R_{>0}$
with correlation kernel $K_{\lm}(x_1,u_1; x_2,u_2)$ given for $x_1,x_2 \in {\{2,3,\dots,2n-2\}}$ and $u_1, u_2 \in \mathbb R_{>0}$, by
\begin{align}
\label{eq_x1} &K_{\lm}(x_1,u_1; x_2;u_2) = \ind{u_{2}< u_{1},x_{2}< x_{1}}\frac{(u_{2}-u_{1})^{x_{1}-x_{2}-1}n^{-\frac{1}{2}(x_{1}-x_{2})}}{(x_{1}-x_{2}-1)!}\\
& +\frac{1}{(2\pi i)^{2}}\oint_{C_{z}[0,\lambda_1+n-x_{2})}dz\oint_{C_{w}[0,x_{1})}dw \frac{\Gamma(-w)}{\Gamma(z+1)}\frac{G_{\lambda}(z+x_{2}-n)}{G_{\lambda}(x_{1}-n-1-w)}\cdot \frac{u_{2}^{z}\,u_{1}^{w}\, n^{-\frac{1}{2}(z+w+1) }}{w+z+x_{2}-x_{1}+1},\notag
\end{align}
where $$G_{\lm}(u) = \G{u+1} \prod_{i=1}^{\infty}
\frac{u+i}{u-\lambda_i+i}= \frac{\G{u+1+n}}{\prod\limits_{i=1}^n (u-\lm_i +i)},$$
with $\lambda=\Delta_n$ for $\X^{'}_{n}$
 and $\lambda=\Delta_{n}\setminus (n-k,k)$ for $\X^{'}_{k,n}$.
The contours $C_z[0,\lm_1+n-x_2)$ and $C_w[0,x_1)$ are as shown in Figure
\ref{fig:contour1}. Both are counter-clockwise, encloses only the integers in the
respective half open intervals $[0,\lm_1+n-x_2)$ and $[0,x_1)$, and are arranged so that $C_z$ and $x_1-x_2-1-C_w$ are disjoint, as in Figure \ref{fig:contour1}.
%such
%that $w+z+x_2-x_1+1$ remains uniformly bounded away from 0 along the contours. Moreover, the contours are arranged in a way that when first integrating w.r.t $w$, the denominator $w+z+x_2-x_1+1$ gives no residues.
\end{thm}

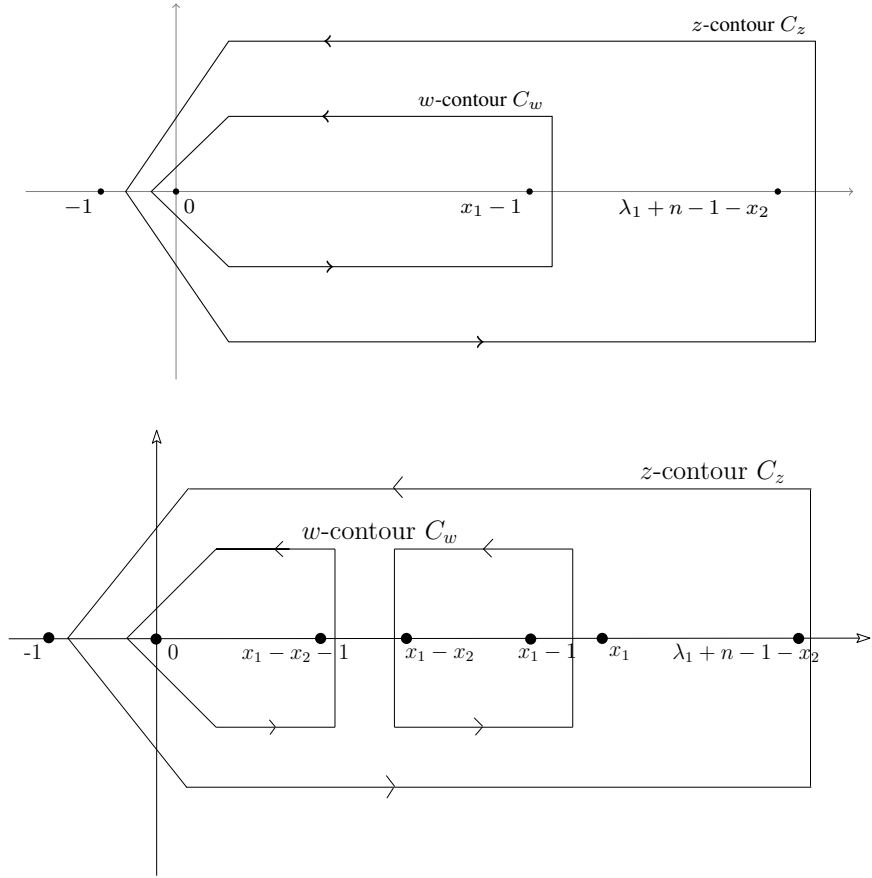
\begin{figure}[t]
\begin{center}
\begin{tikzpicture}
% The axes
\draw[help lines,->] (-2,0) -- (9,0) coordinate (xaxis); \draw[help lines,->]
(0,-2.5) -- (0,2.5) coordinate (yaxis);

% The z-contour
\draw [decoration={markings, mark=at position 4.5cm with {\arrow[line
width=0.75pt]{<}}, mark=at position 14cm with {\arrow[line width=0.75pt]{<}}},
postaction={decorate}] (8.5,-2) -- (0.7,-2) -- (-0.67,0) -- (0.7,2) -- (8.5,2)
node[above left] {$z$-contour $C_z$} -- (8.5,-2);

% The w-zontour
\draw [decoration={markings, mark=at position 3cm with {\arrow[line
width=0.75pt]{<}}, mark=at position 8.5cm with {\arrow[line width=0.75pt]{<}}},
postaction={decorate}] (5,-1) -- (0.7,-1) -- (-0.33,0) -- (0.7,1) -- (5,1)
node[above left] {$w$-contour $C_w$} -- (5,-1);

% The labels
\draw [black,fill=black] (0,0) circle (1pt) node[below right] {$0$}; \draw
[black,fill=black] (-1,0) circle (1pt) node[below left] {$-1$}; \draw
[black,fill=black] (4.7,0) circle (1pt) node[below left] {$x_1-1$}; \draw
[black,fill=black] (8.0,0) circle (1pt) node[below left] {$\lambda_1+n-1-x_2$};
\end{tikzpicture}

\bigskip

\includegraphics[scale=0.7]{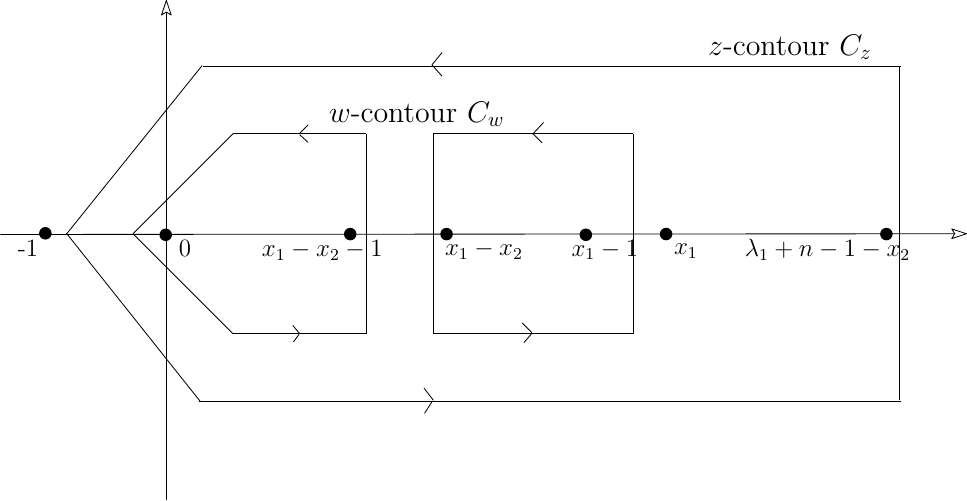}

\caption{\small{The contours in the statement of Theorem \ref{thm:jumps},  correspond to the case $2\le x_1\le x_2$ (top) and $2\le x_2< x_1$ (bottom). In the bottom picture, $w$-contour splits into two components, in order to guarantee that $C_z$ and $x_1-x_2-1-C_w$ are disjoint}}
\label{fig:contour1}
\end{center}
\end{figure}

\begin{remark} \label{Remark_integrand_simplification}
 The $\Gamma$--functions in double-contour integrals are used for the convenience of notations, but the integral is, in fact, a rational function of $z$ and $w$ multiplied by $u_{2}^{z}\,u_{1}^{w}\, n^{-\frac{1}{2}(z+w+1)}$. Indeed, we have
 \begin{equation}
 \label{eq_x2}
  \frac{\Gamma(-w)}{\Gamma(z+1)}\frac{G_{\lambda}(z+x_{2}-n)}{G_{\lambda}(x_{1}-n-1-w)}= \frac{(z+1)(z+2)\cdots(z+x_2)}{\prod\limits_{i=1}^n (z+x_{2}-n-\lm_i +i)} \cdot  \frac{\prod\limits_{i=1}^n (-w+x_{1}-n-1-\lm_i +i)}{(-w)(1-w)\cdots(x_1-1-w)}.
 \end{equation}
\end{remark}

%\begin{remark}\label{rem:contour}
%When $u_1$ or $u_2$ equals 0, $K_{\lm}$ is to be understood as the limit as
%$u_1$ or $u_2$ tends to 0. But for the purpose of computation it suffices to consider generic $u_{1},u_{2}$.
%\end{remark}
\begin{proof}[Proof of Theorem \ref{thm:jumps}]
This is a corollary of \cite[Theorem 1.5]{GR}. In the notation of \cite{GR}, we choose:
$$t_1=1-\frac{u_1}{n^{\frac{1}{2}}};\ t_2=1-\frac{u_2}{n^{\frac{1}{2}}};\ x_1^{[GR]}=x_1-n;\ x_2^{[GR]}=x_2-n;\qquad \lm=\Delta_{n}\;\text{or}\; \Delta_{n}\setminus (n-k,k).$$
Since our match of parameters involves rescaling of the real spatial coordinate, we also need to use the following Lemma \ref{lemma:rescaling}, whose proof we omit.
\end{proof}

\begin{lemma}\label{lemma:rescaling}
Let $Y_{n}$ be a determinantal point process on $\Z\times (0,1)$ with reference measure $\#_{\Z}\times \Lb((0,1))$ and correlation kernel $K_{n}$. For $c_{n}\in \Z$, $\alpha\in \R$, we define the rescaled and shifted point process
$$Y^{'}_{n}=\{(x-c_{n},n^{\alpha}u)\mid (x,u)\in Y_{n}\}.$$
Then, the correlation kernel of $Y^{'}_{n}$ with reference measure $\#_{\Z}\times \Lb(\R_{> 0})$ is $$K^{'}_{n}(x_{1},u_{1};\, x_{2},u_{2})=n^{-\alpha}K_{n}\Bigg(x_{1}+c_{n},\frac{u_{1}}{n^{\alpha}};\, x_{2}+c_{n},\frac{u_{2}}{n^{\alpha}}\Bigg).$$
\end{lemma}

\begin{remark}\label{rem:unifconvergence}
We can extend $\X_{n}^{'}$ and $\X_{k,n}^{'}$ to point processes on $ {\{2,3,\dots,2n-2\}}\times \R_{\ge 0}$, in such a way that almost surely there is no particle at $(x,0)$ for any $x\in  {\{2,3,\dots,2n-2\}}$. Moreover, the kernel $K_{\lambda}$ can be extended to the same space based on its regular behavior at $u_1$ or $u_2=0$.

In order to see that the double contour integral in the definition of the kernel is well-behaved as $u_1\to 0$ or $u_2\to 0$, we expand the integral as a sum of residues at the poles inside the integration contours. We start from the case $2\le x_1 \le x_2$ with contours of the top panel of Figure \ref{fig:contour1}. In this case the denominator $w+z+x_2-x_1+1$ is never singular inside the integration contours, the $w$--residues are at simple poles at finitely many non-negative integers and $z$--residues are also at simple poles at finitely many non-negative integers. The residue at $w=m_1$, $z=m_2$ for $m_1,m_2\in\mathbb Z_{\ge 0}$ has the form $c(m_1,m_2)\cdot u_1^{m_1} u_2^{m_2}$, where $c(m_1,m_2)$ does not depend on $u_1,u_2$. Hence, the residue is continuous as $u_1, u_2\to 0$ and so is the double integral.

Proceeding to the case $2\le x_2< x_1$ with contours of the bottom panel of Figure \ref{fig:contour1}, the additional feature is the potential singularity of the denominator $w+z+ x_2 -x_1+1$. Let us first compute the $w$--integral as a sum of the residues, getting a sum of the terms of the form \mbox{$u_1^{m_1}\times$ (one-dimensional $z$-integral)} with non-negative values for $m_1$. Up to the factors not depending on $u_1$ or $u_2$, the corresponding $z$--integral has the form
$$
  \oint_{C_{z}[0,\lambda_1+n-x_{2})} \frac{\Gamma(z+x_2+1)}{\Gamma(z+1)}\cdot \frac{1}{\prod\limits_{i=1}^n(z+x_2-\lambda_i+i-n)} \cdot \frac{u_{2}^{z} \, n^{-\frac{1}{2}z }}{z+m_1+x_{2}-x_{1}+1} dz.
$$
The simple poles of the last integral lead to $u_2$--dependence of the form $u_2^{m_2}$, $m_2\ge 0$, which is again continuous at $u_2=0$. It would be more complicated, if the integral had a double pole: the residue at such a pole would have involved the $z$--derivative of $u_2^z$, which leads to the appearance of $\ln(u_2)$ factor that is singular at $u_2=0$. However, we claim that there is no double pole. Indeed, using \eqref{eq_x2}, $m_1\notin \{x_1-n-1-\lambda_i+i\}_{i=1}^n$. Therefore, for the pole of $\frac{1}{z+m_1+x_{2}-x_{1}+1}$, we have $(-m_1-x_2+x_1-1)\notin \{-x_2+\lambda_i-i+n\}_{i=1}^n$. Hence, this pole is outside the set of poles of $\frac{1}{\prod_{i=1}^n(z+x_2-\lambda_i+i-n)}$.
\end{remark}

\subsection{Anti-symmetric GUE}\label{subsec:ague}  Let us recall a statement from \cite[Theorems 2.9 and 3.3]{FN}:

\begin{prop} For $M$ in Definition \ref{def:cornerprocess} and $l=1,2,\dots$ we have:\label{prop:aGUE}

\begin{enumerate}[(a)]
\item The $l^\text{th}$ corner of $M$ has $l$ real eigenvalues.

\item The eigenvalues come in pairs, i.e, $\lambda$ is an eigenvalue of the $l^\text{th}$ corner if and only if $-\lambda$ is an eigenvalue of the $l^\text{th}$ corner. When $l$ is odd, $0$ is necessarily an eigenvalue of the $l^{\text{th}}$ corner.

\item The eigenvalues of the $l^\text{th}$ corner interlace with the eigenvalues of the $(l+1)^\text{st}$ corner a.s, which means that
\begin{equation}
\lambda^{l+1}_{j+1}\le \lambda^{l}_{j}\le \lambda^{l+1}_{j},
\label{eq_x3}
\end{equation}
where $j=1,2,\dots,l$ and $\lambda^{l}_{j}$ denotes the $j^\text{th}$ largest eigenvalue of the $l^\text{th}$ corner.

\item \label{eq_aGUE_cond} Conditionally on the positive eigenvalues of the $l^\text{th}$ corner, the positive eigenvalues of the $i^\text{th}$ corners, $i=2,...,l-1$, are jointly distributed uniformly on the polytope determined by the interlacing inequalities \eqref{eq_x3}.

\item  The joint density of positive eigenvalues $\bigl(\lambda^{l}_{1}\ge \dots\ge \lambda^{l}_{\lfloor\frac{l}{2}\rfloor}\bigr)$ of the $l^\text{th}$ corner  is
\begin{align*}
&p(\lambda^{l}_{1},...,\lambda^{l}_{\lfloor \frac{l}{2}\rfloor})=
\begin{dcases}
C_{l} \prod_{1\le i< j \le \frac{l}{2}} \left( (\lambda^{l}_i)^2-(\lambda^l_j)^2\right)\, \prod_{j=1}^{\frac{l}{2}}e^{-(\lambda^{l}_{j})^{2}}, & \text{l}\;\; \text{is}\;\;\text{even}; \\
\quad\\
C_{l}\prod_{1\le i< j \le \frac{l-1}{2}} \left( (\lambda^{l}_i)^2-(\lambda^l_j)^2\right)\, \prod_{j=1}^{\frac{l-1}{2}}\left[\lambda^{l}_{j} \, e^{-(\lambda^{l}_{j})^{2}}\right], & \text{l}\;\; \text{is}\;\;\text{odd},
\end{dcases}
\end{align*}
where $C_{l}$ is an explicitly known normalization constant.
\end{enumerate}

\end{prop}

Note that the above five properties uniquely characterize the joint distribution of $\{\lambda^{k}_{j}\}_{1\le j\le k< \infty}$. An alternative characterization is through the correlation functions of the corresponding point process  $\X$:

\begin{thm}[{\cite[Proposition 4.3]{FN}}]\label{thm:ague}
The aGUE corners process $\X$ of Definition \ref{def:cornerprocess} is a determinantal point process with correlation kernel $K(x,u;y,v)$  (with respect to the reference measure $\#_{\Z} \otimes \Lb(\R_{\ge 0})$), such that for $x\ge y$,
\begin{equation}\label{eq_Cor_extended}K(x,u;y,v)=e^{-u^{2}}\sum_{l=1}^{[y/2]}\frac{H_{x-2l}(u)H_{y-2l}(v)}{N_{y-2l}},\end{equation}
and for $x< y$,
\begin{equation} K(x,u;y,v)=-e^{-u^{2}}\sum_{l=-\infty}^{0}\frac{H_{x-2l}(u)H_{y-2l}(v)}{N_{y-2l}}.\end{equation}
Here $H_{j}(u)$ is the $j^\text{th}$ ``physicist's Hermite polynomial'' with leading coefficient $2^j$, so that
$$\int_{\R}H_{i}(u)H_{j}(u)e^{-u^{2}}du=2 \delta_{ij}N_j, \qquad N_j=j!2^{j-1}\sqrt{\pi}.$$
\end{thm}
\begin{remark}
 The $x=y$ case of \eqref{eq_Cor_extended} differs from the kernel of \eqref{eq_Cor_intro} by the factor $e^{-\tfrac12 u^2+\tfrac12 v^2}$, which cancels out when we compute determinants.
\end{remark}

\section{Local Limit of Standard Staircase Shaped Tableaux} \label{sec:calculate}

The aim of this section is to investigate the scaling limit near the bottom-left corner for the uniformly random standard Young tableaux of shapes $\Delta_n$ and $\Delta_{n}\setminus(n-k,k)$ as $n\to\infty$.

\subsection{Statement of the result} We use the $(l,m)$ coordinate system for SYT, as in Section \ref{subsec:rotation}. We recall that $\X$ is the point process corresponding to the corners of aGUE as in Definition \ref{def:cornerprocess}.  Let $\X(l,m)$  denote the $m^\text{th}$ particle of $\X$ in the level $l$, sorted in the increasing order; in other words, $\X(l,m)$ is the $m^\text{th}$ positive eigenvalue of $l\times l$ corner of the infinite aGUE matrix $M$ (with $m=1$ corresponding to the smallest positive and $m=\lfloor \tfrac{l}{2}\rfloor$ to the largest positive eigenvalue).

Further, we set $\X^{'}_{k}$ to be the point process of Definition \ref{def:kernel} with an extra particle inserted at $(2k,0)$. We let $\X^{'}_{k}(l,m)$ denote the $m^\text{th}$ particle of $\X^{'}_{k}$ in the level $x=l$, sorted in the increasing order.

\begin{thm}\label{thm:syt} For each $n\in\mathbb Z_{\ge 2}$, let $T_{\Delta_{n}}(l,m)$ and $T_{\Delta_{n}\setminus(n-k,k)}(l,m)$ be uniformly random standard Young tableaux of shapes $\Delta_n$ and ${\Delta_{n}\setminus(n-k,k)}$, respectively, in the coordinate system of Section \ref{subsec:rotation}. Recall $N={n\choose 2}$. Then for each fixed $k$, as $n\to\infty$ we have
\begin{equation}
\left\{n^{\frac{1}{2}}\left(1-\frac{T_{\Delta_{n}}(l,m)}{N}\right)\right\}_{l\ge 2,\,\, 1\le m \le\lfloor \frac{l}{2}\rfloor }\longrightarrow \bigl\{\X(l,m)\bigr\}_{l\ge 2,\,\, 1\le m \le\lfloor\frac{l}{2}\rfloor}\;
 \qquad \text{ and }
 \label{eq_limit_1}
\end{equation}
\begin{equation}
\left\{n^{\frac{1}{2}}\left(1-\frac{T_{\Delta_{n}\setminus(n-k,k)}(l,m)}{N}\right)\right\}_{l\ge 2,\,\, 1\le m \le\lfloor\frac{l}{2}\rfloor}\longrightarrow \bigl\{\X^{'}_{k}(l,m)\bigr\}_{l\ge 2,\,\, 1\le m \le\lfloor \frac{l}{2}\rfloor},
\label{eq_limit_2}
\end{equation}
in the sense of convergence of finite-dimensional distributions.
\end{thm}
\begin{remark} \label{remark_conditional}
 The SYT $T_{\Delta_{n}\setminus(n-k,k)}$ can be thought as $T_{\Delta_{n}}$ conditioned on its largest entry being at $(n-k,k)$, and in \eqref{eq_limit_2}, $T_{\Delta_{n}\setminus(n-k,k)}(2k,0)$ is set to be $N$, corresponding to $\X(2k,1)=0$ on the right. Hence, Theorem \ref{thm:syt} suggests a conjecture: $\X^{'}_{k}$ should have the same law as $\X$ conditioned on having a particle at $(2k,0)$, i.e.\ on $\X(2k,1)=0$. Note that additional efforts are needed to prove this, because the topology of the convergence in Theorem \ref{thm:syt} is not strong enough to make conclusions about the conditional distributions.
\end{remark}

The rest of this section is devoted to the proof of Theorem \ref{thm:syt}. We first couple the two types of SYT with the uniformly random PYT of the same shape, and induce two point processes $\X_{n}^{'}, \X_{k,n}^{'}$ on $\Z_{\ge 0}\times \R_{\ge 0}$ respectively from these two PYT, as in Definition \ref{def:pointprocess}. Then we prove that $\X_{n}^{'}\rightarrow \X^{'}$ and $\X_{k,n}^{'}\rightarrow \X^{'}_{k}$ as $n\to\infty$, where $\X^{'}$ and $\X^{'}_{k}$ are determinantal point processes on $\Z_{\ge 0}\times \R_{\ge 0}$ with explicit correlation kernels given by double contour integrals. Next, we identify $\X^{'}$ with $\X$, the corners process of aGUE of Definition \ref{def:cornerprocess}.

\subsection{Coupling with PYT}\label{Section_SYT_to_PYT}
The first step of the proof is to introduce a coupling between SYT and PYT of the same shape and to show that under this coupling the difference between random SYT and PYT is negligible in the asymptotic regime of our interest.

Let $T_{\Delta_n}$ and $\T_{\Delta_{n}}$ be uniformly random SYT and PYT, respectively, of shape $\Delta_n$. We couple these two tableaux in the following way: for $N={n\choose 2}$, let $(P_{1}\ < \ P_{2}\ <\cdots<\ P_{N})$
be a uniformly random point on the simplex $\{(x_{1},\cdots,x_{N})\mid 0< x_1 < x_2< \dots < x_N < 1\}$. Given a uniformly random SYT $T_{\Delta_{n}}$, we replace entry $l$ in $T_{\Delta_{n}}$ by $P_{l}$ for $l=1,2,\cdots ,N$; it is straightforward to check that the result is a uniformly random PYT $\T_{\Delta_{n}}$. In the opposite direction,  $T_{\Delta_{n}}$ is reconstructed by $\T_{\Delta_{n}}$ by replacing the $l$th largest entry in $\T_{\Delta_{n}}$ by $l$ for each $l=1,2,\dots,N$.

\begin{lemma}\label{lemma:coupling}
Fix any pair of integers $(i,j)$ with $i\ge 2$ and $1\le j \le \lfloor \tfrac{i}2\rfloor$.
Using the coupling described above, we have for each $\delta> 0$
$$n^{1-\delta}\left|\T_{\Delta_{n}}(i,j)-\frac{T_{\Delta_{n}}(i,j)}{N}\right|\rightarrow 0$$ in probability, as $n\rightarrow \infty$.
\end{lemma}

\begin{proof}
Let $T_{\Delta_{n}}(i,j)=\xi\in \{1,2,\dots,N\}$, where $\xi$ is a random variable. Then, by the construction of the  coupling, $\T_{\Delta_{n}}(i,j)=P_{\xi}$.
By Chebyshev's inequality,
\begin{equation}
\begin{split}
& \quad \pr{n^{1-\delta}\left|P_{\xi}-\frac{\xi}{N}\right|> \epsilon}\\
& \le \left(\frac{n^{1-\delta}}{\epsilon}\right)^{2}\E{\left(P_{\xi}-\frac{\xi}{N}\right)^{2}} \\
 & = \frac{n^{2-2\delta}}{\epsilon^{2}}\sum_{j=1}^{N}\E{\left(P_{\xi}-\frac{\xi}{N}\right)^{2}\, \Big| \, \xi=l}\cdot \pr{\xi=l}
\end{split}
\label{eq_coupling_1}
\end{equation}
For a fixed $l$, the random variable $P_{l}$ is distributed according to Beta distribution with parameters $l$ and $N+1-l$, which has mean $\frac{l}{N+1}$
and variance (see, e.g., \cite[Chapter 25]{JKB}).
\begin{equation}
\E{\left(P_{l}-\frac{l}{N+1}\right)^{2}}=\frac{l(N+1-l)}{(N+1)^{2}(N+2)}\le \frac{\left(\frac{N+1}{2}\right)^{2}}{(N+1)^2(N+2)}=\frac{1}{4(N+2)}.
\label{eq_coupling_2}
\end{equation}
Thus,
\begin{equation}
\E{\left(P_{l}-\frac{l}{N}\right)^{2}}=\E{\left(P_{l}-\frac{l}{N+1}\right)^{2}}+\left(\frac{l}{N(N+1)}\right)^2 \le \frac{1}{4(N+2)}+\frac{1}{(N+1)^2}\le \frac{1}{N}.
\label{eq_coupling_3}
\end{equation}
Since $N=\frac{n(n-1)}{2}$, \eqref{eq_coupling_1} and \eqref{eq_coupling_3} imply that as $n\rightarrow \infty$
$$\pr{n^{1-\delta}\left|P_{\xi}-\frac{\xi}{N}\right|> \epsilon}\le \frac{n^{2-2\delta}}{\eps^2 N} \rightarrow 0. \qedhere$$
\end{proof}

In a similar way, we couple uniformly random SYT $T_{\Delta_{n}\setminus(n-k,k)}$ and PYT $\T_{\Delta_{n}\setminus(n-k,k)}$  of shape $\Delta_n\setminus(n-k,k)$: we let $(P_{1}\ < \ P_{2}\ <\cdots<\ P_{N-1})$
be a uniformly random point on the simplex $\{(x_{1},\cdots,x_{N-1})\mid 0< x_1 < x_2< \dots < x_{N-1} < 1\}$. Given a uniformly random SYT $T_{\Delta_{n}\setminus(n-k,k)}$, we replace entry $k$ in $T_{\Delta_{n}\setminus(n-k,k)}$ by $P_{k}$ for $k=1,2,\cdots ,N-1$; it is straightforward to check that the result is a uniformly random PYT $\T_{\Delta_{n}\setminus(n-k,k)}$. Repeating the proof of Lemma \ref{lemma:coupling} we arrive at:
\begin{lemma}\label{lemma:coupling_2}
Fix any pair of integers $(i,j)$ with $i\ge 2$ and $1\le j \le \lfloor \tfrac{i}2\rfloor$.
Using the coupling described above, we have for each $\delta> 0$
$$n^{1-\delta}\left|\T_{\Delta_{n}\setminus(n-k,k)}(i,j)-\frac{T_{\Delta_{n}\setminus(n-k,k)}(i,j)}{N}\right|\rightarrow 0$$ in probability, as $n\rightarrow \infty$.
\end{lemma}

\subsection{Limiting Processes}\label{subsec:limit} In this section we analyze asymptotic behavior of the point processes associated to random PYT
$\T_{\Delta_{n}}$ and $\T_{\Delta_{n}\setminus(n-k,k)}$ by the procedure of Section \ref{subsec:embd}.

\begin{thm}\label{thm:def1}
The point process $\X^{'}_{n}$ (corresponding to random PYT $\T_{\Delta_{n}}$ via Definition \ref{def:pointprocess}) converges weakly as $n\to\infty$ to a point process $\X^{'}$ on $\Z_{\ge 2}\times \R_{\ge 0}$.
$\X^{'}$ is a determinantal point process (with respect to reference measure $\#_{\Z_{\ge 2}}\otimes \Lb(\R_{\ge 0})$) with correlation kernel
\begin{align*}
 K^{'}&(x_{1},u_{1};x_{2},u_{2})= I_{\{u_{2}< u_{1},x_{2}< x_{1}\}}\frac{(u_{2}-u_{1})^{x_{1}-x_{2}-1}}{(x_{1}-x_{2}-1)!} \\
&+\frac{1}{(2\pi i)^{2}}\oint_{C_{z}[0,\infty)}dz\oint_{C_{w}[0,x_{1})}dw \frac{\Gamma(-w)}{\Gamma(z+1)}\frac{\Gamma(z+x_{2}+1)}{\Gamma(x_{1}-w)}\frac{\Gamma(-\frac{z+x_{2}}{2})}{\Gamma(\frac{-x_{1}+w+1}{2})}\frac{u_{1}^{w}u_{2}^{z}}{w+z+x_{2}-x_{1}+1}.
\end{align*}
Here $u_{1},u_{2}\in \R_{>0}$ and $x_{1},x_{2}\in \Z_{\ge 2}$; when $u_1$ or $u_2$ is equal to $0$, the kernel is defined as the limit as $u_1$ or $u_2$ tends to 0. $C_{z}[0,\infty), C_{w}[0,x_{1})$ are counterclockwise--oriented contours which enclose the integers in $[0,\infty)$ and $[0,x_{1})$, respectively, and are arranged so that $C_z$ and $x_1-x_2-1-C_w$ are disjoint, as in Figure \ref{Figure-unboundedcontour}.
\end{thm}

\begin{figure}[htpb]
    \begin{center}
        \includegraphics[scale=0.6]{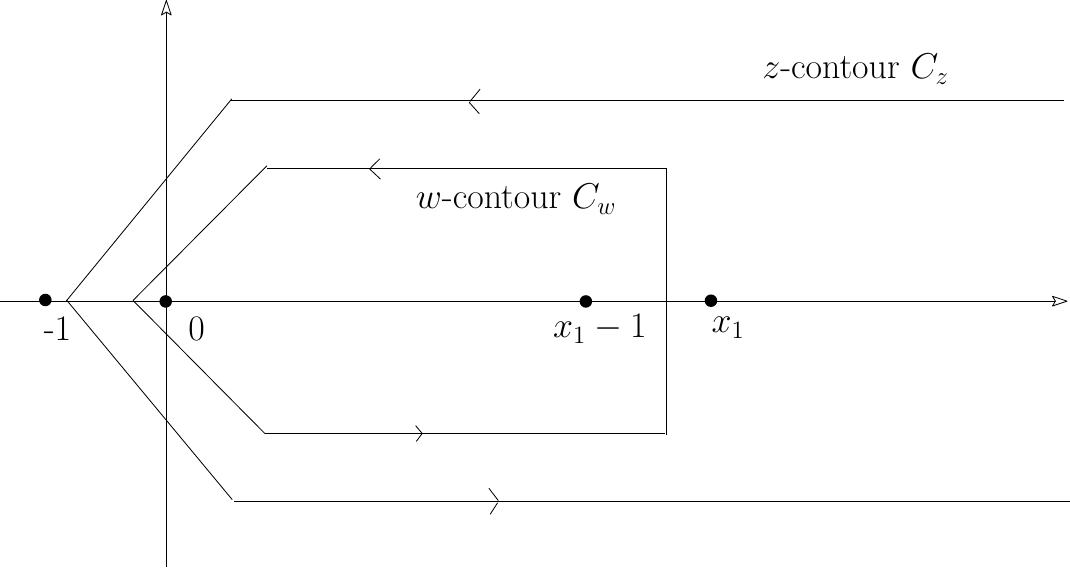}

        \medskip

        \includegraphics[scale=0.6]{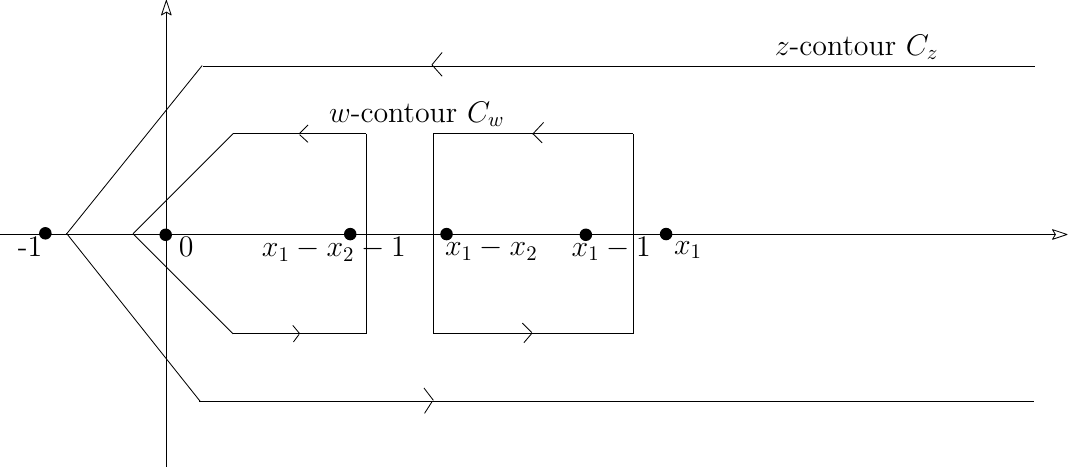}
        \caption{\small{The contours in the statements of Theorem \ref{thm:def1} and \ref{thm:def2}, for the case $1\le x_1\le x_2$ and $1\le x_2<x_1$ respectively. }\label{Figure-unboundedcontour}}
    \end{center}
\end{figure}

\begin{thm} \label{thm:def2}
The point process $\X^{'}_{k,n}$  (corresponding to random PYT $\T_{\Delta_{n}\setminus (n-k,k)}$ via Definition \ref{def:pointprocess}) converges weakly to a point process $\X^{'}_{k}$ on $\Z_{\ge 2}\times \R_{\ge 0}$.
$\X^{'}_{k}$ is a determinantal point process (with respect to reference measure $\#_{\Z_{\ge 2}}\otimes \Lb(\R_{\ge 0})$) with correlation kernel given by
\begin{align*}
K^{'}_{k}&(x_{1},u_{1};x_{2},u_{2})=  I_{\{u_{2}< u_{1},x_{2}< x_{1}\}}\frac{(u_{2}-u_{1})^{x_{1}-x_{2}-1}}{(x_{1}-x_{2}-1)!} \\
&+ \frac{1}{(2\pi i)^{2}}\oint_{C_{z}[0,\infty)}dz\oint_{C_{w}[0,x_{1})}dw \frac{\Gamma(-w)}{\Gamma(z+1)}\frac{\Gamma(z+x_{2}+1)}{\Gamma(x_{1}-w)}\frac{\Gamma(-\frac{z+x_{2}}{2})}{\Gamma(\frac{-x_{1}+w+1}{2})}\\ &\quad \times \frac{z+x_{2}-2k}{z+x_{2}-2k+1}\cdot \frac{x_{1}-w-2k}{x_{1}-w-2k-1}\cdot \frac{u_{1}^{w}u_{2}^{z}}{w+z+x_{2}-x_{1}+1}.
\end{align*}
Here $u_{1},u_{2}\in \R_{>0}$ and $x_{1},x_{2}\in \Z_{\ge 0}$; the value of $K^{'}_{k}$ when $u_1$ or $u_2$ equals 0 is to be understood as the limit as $u_1$ or $u_2$ tends to 0. The contours $C_{z}[0,\infty)$ and $C_{w}[0,x_{1})$ are the same as in Theorem \ref{thm:def1}.
\end{thm}

In the proof we use a known asymptotic property of the Gamma-function:
\begin{equation} \label{eq_Gamma_assy}
 \frac{\Gamma(y+m)}{(m-1)!}=m^{y}(1+O(m^{-1})), \qquad m\to + \infty,
\end{equation}
where $O(m^{-1})$ term is uniform as long as $y$ is uniformly bounded.

%We are going to use two well-known properties of the Gamma function:

%\begin{enumerate}

%\item $\displaystyle \frac{\Gamma(y+m)}{(m-1)!}=m^{y}(1+O(m^{-1}))$, where $m\ge 1$ and the big O term is uniform in $m$ as long as $y$ is bounded and bounded away from negative integers.

%\item   For $y\notin \{0,-1,-2...\}$, we have
%$\displaystyle \Gamma(1-y)\Gamma(y)=\frac{\pi}{\sin(\pi y)}.$
%\end{enumerate}

\begin{proof}[Proof of Theorem \ref{thm:def1}]

We show the convergence in distribution by verifying conditions of Lemma \ref{lemma:conv}, namely, we show that the correlation kernel $K^{'}_{n}(x_{1},u_{1};x_{2},u_{2}) n^{\frac{x_{1}-x_{2}}{2}}$ of $\X^{'}_{n}$, where $K^{'}_n$ is given by \eqref{eq_x1}, converges to $K^{'}$ uniformly on compact subsets of $x_{1},x_{2}\in \Z_{\ge 2}$ and $u_{1},u_{2}\in \R_{\ge 0}$. Note that the just introduced multiplication by $n^{\frac{x_{1}-x_{2}}{2}}$ does not change the value of the correlations functions $\det[K^{'}_{n}(x_{i},u_{i};x_{j},u_{j})]$. Let us fix arbitrary $x_1,x_2\in \Z_{\ge 2}$, strictly positive $u_1$, $u_2$, and analyze the asymptotic behavior of the double contour integral in \eqref{eq_x1}.

The first step is to deform the contours of integration from the ones of Figure \ref{fig:contour1} to the ones of Figure \ref{Figure-unboundedcontour}. Using \eqref{eq_x2}, the $z$--dependent factors of the integrand are
 \begin{equation}
 \label{eq_x4}
 \frac{(z+1)(z+2)\cdots(z+x_2)}{\prod\limits_{i=1}^n (z+x_{2}-2n+2i)} \cdot \left(\frac{u_2}{\sqrt{n}}\right)^z \cdot \frac{1}{w+z+x_{2}-x_{1}+1}.
 \end{equation}
 The form of \eqref{eq_x4} implies that there are no $z$--poles of the integrand to the right from the contour of Figure \ref{fig:contour1}; hence, the value of the  integral does not change in the deformation. For each fixed $u_2$ and large enough $n$, on the new contours of Figure \ref{Figure-unboundedcontour}, the expression \eqref{eq_x4} rapidly decays as $\Re z\to +\infty$ (uniformly in $n$). This observation allows us to control the part of the integral corresponding to large $\Re z$, and it remains to study the $n\to\infty$ asymptotics of the integrand for finite $z$ and $w$.

We use \eqref{eq_Gamma_assy} to compute the asymptotic behavior of $\prod_{i=1}^n$ in \eqref{eq_x4}:
\begin{align*}
&\prod_{i=1}^{n}(z+x_{2}-2n+2i)
=(-1)^{n}\prod_{i=1}^{n}(-z-x_{2}+2n-2i)\\
&=(-1)^{n}2^{n}\prod_{i=1}^{n}\left(-\frac{z-x_{2}}{2}+n-i\right)
=(-1)^{n}2^{n}\frac{\Gamma(-\frac{z+x_{2}}{2}+n)}{\Gamma(-\frac{z+x_{2}}{2})}\\
&={(-1)^{n}}2^{n} (n-1)! \frac{n^{-\frac{z+x_{2}}{2}}}{\Gamma(-\frac{z+x_{2}}{2})}\left(1+O(n^{-1})\right).
\end{align*}
Similarly,
$$\prod_{i=1}^{n}(x_{1}-w-2n+2i-1)=(-1)^{n}2^{n}(n-1)! \frac{n^{\frac{-x_{1}+w+1}{2}}}{\Gamma(\frac{-x_{1}+w+1}{2})}\left(1+O(n^{-1})\right) .$$
Thus, for $\lambda=\Delta_n$, we have
$$\frac{G_{\lambda}(z+x_{2}-n)}{G_{\lambda}(x_{1}-n-1-w)}=
\frac{\Gamma(z+x_{2}+1)}{\Gamma(x_{1}-w)}\frac{\Gamma(-\frac{z+x_{2}}{2})}{\Gamma(\frac{-x_{1}+w+1}{2})}n^{\frac{-x_{1}+w+1}{2}}n^{\frac{z+x_{2}}{2}} \left(1+O(n^{-1})\right).$$
Hence, the integrand in the double contour integral part of $K^{'}_{n}(x_{1},u_{1};x_{2},u_{2}) n^{\frac{x_{1}-x_{2}}{2}}$,  converges as $n\to\infty$ to
$$\frac{\Gamma(-w)}{\Gamma(z+1)}\frac{\Gamma(z+x_{2}+1)}{\Gamma(x_{1}-w)}\frac{\Gamma(-\frac{z+x_{2}}{2})}{\Gamma(\frac{-x_{1}+w+1}{2})}\frac{u_{1}^{w}u_{2}^{z}}{w+z+x_{2}-x_{1}+1}.$$
Combining with the straightforward asymptotics of the indicator term, we conclude that for all strictly positive $u_1$ and $u_2$,
$$
 \lim_{n\to\infty} K^{'}_{n}(x_{1},u_{1};x_{2},u_{2}) n^{\frac{x_{1}-x_{2}}{2}}=K^{'}_{n}(x_{1},u_{1};x_{2},u_{2}).
$$
Clearly, the convergence is uniform over $u_1,u_2$ in compact subsets of $(0,+\infty)$. It remains to show that there is no explosion for $u_1$ or $u_2$ near $0$, which is done in exactly the same way as we did in Remark \ref{rem:unifconvergence}, i.e.\ by interpreting the integral as a sum of residues at simple $z-$ and $w-$poles at non-negative integers.
\end{proof}

\begin{proof}[Proof of Theorem \ref{thm:def2}] The proof is the same as for Theorem \ref{thm:def1}, with only some slight difference in calculation. This time, for $\lambda=\Delta_{n}\setminus(n-k,k)$ we have
\begin{align*}
G_{\lambda}(z+x_{2}-n)=\frac{\Gamma(z+x_{2}+1)}{\prod\limits_{i=1}^{n}(z+x_{2}-2n+2i)}\cdot\frac{z+x_{2}-2k}{z+x_{2}-2k+1},
\end{align*}
\begin{align*}
G_{\lambda}(x_{1}-n-1-w)=\frac{\Gamma(x_{1}-w)}{\prod\limits_{i=1}^{n}(x_{1}-w-2n+2i-1)}\cdot\frac{x_{1}-w-2k-1}{x_{1}-w-2k}.
\end{align*}
The additional fractions $\frac{z+x_{2}-2k}{z+x_{2}-2k+1}$ and $\frac{x_{1}-w-2k-1}{x_{1}-w-2k}$  (as compared to $\lambda=\Delta_n$ case) propagate to the final answer without changes.
\end{proof}

\subsection{Connection to Anti-symmetric GUE}\label{subsec:check}

In this section we prove that the distribution of the limiting process of Theorem \ref{thm:def1} coincides with the corners process of aGUE.

\begin{prop} \label{Proposition_kernel_section}
\label{eq_kernel_section}
$K^{'}(2k,u_{1};2k,u_{2})$ of Theorem \ref{thm:def1} for $u_{1},u_{2}\in \R_{\ge 0}$, $k\in\mathbb Z_{>0}$, is equal to
\begin{align}\label{eq_kernel_section_1}
2^{k+1}\sum_{i=0}^{\infty}\sum_{j=0}^{k-1}&(-1)^{i+k}\frac{u_{1}^{2j}}{(2k-2j-1)! (2j)! \, \Gamma(j+\frac{1}{2}-k)}\cdot\frac{u_{2}^{2i} \cdot \prod_{a=1}^{k}(2i+2a-1) }{i!}\cdot\frac{1}{2j+2i+1}.
\end{align}
\end{prop}
\begin{remark} \label{Remark_match_small_k}When $k=1$, we simplify $K^{'}(2,u_{1};2,u_{2})=\frac{2}{\sqrt{\pi}}e^{-u_{2}^{2}}$. There's only one particle on level $\{2\}\times \R_{\ge 0}$ which corresponds to the limit of the entry of the SYT at lower corner, with coordinate $(n-1,1)$.

When $k=2$, we simplify $K^{'}(4,u_{1};4,u_{2})=\frac{1}{\sqrt{\pi}}[(1-2u_{1}^{2})(1-2u_{2}^{2})+2]e^{-u_{2}^{2}}$. There are two particles on level $\{4\}\times \R_{\ge 0}$ which correspond to the two entries of the SYT with coordinates $(n-3,1)$ and $(n-2,2)$.
\end{remark}
\begin{proof}[Proof of Proposition \ref{Proposition_kernel_section}]
$K^{'}(2k,u_{1};2k,u_{2})$ of Theorem \ref{thm:def1} is given by
\begin{align*}
    \frac{1}{(2\pi \ii)^{2}}\oint_{C_{z}[0,\infty)}dz\oint_{C_{w}[0,2k)}dw \frac{\Gamma(-w)}{\Gamma(z+1)}\frac{\Gamma(z+2k+1)}{\Gamma(2k-w)}\frac{\Gamma(-\frac{z+2k}{2})}{\Gamma(\frac{-2k+w+1}{2})}\frac{(u_{1})^{w}(u_{2})^{z}}{w+z+1},
\end{align*}
The $w$--integral is evaluated as the sum of residues at simple poles at even integers $w=0,2,4\dots,2k-2$ and we have
\begin{multline*}
 \mathrm{Res}_{w=2j} \left[\frac{\Gamma(-w)}{\Gamma(2k-w)\Gamma(\frac{-2k+w+1}{2})} \cdot
\frac{(u_{1})^{w}}{w+z+1}\right]\\=  \mathrm{Res}_{w=2j} \left[\frac{1}{\Gamma(\frac{-2k+w+1}{2}) (-w)(-w+1)\cdots(-w+2k-1)} \cdot
\frac{(u_{1})^{w}}{w+z+1}\right]\\= -\left[\frac{1}{\Gamma(\frac{-2k+2j+1}{2})\,
(2k-1-2j)!(2j)!
} \cdot
\frac{(u_{1})^{2j}}{2j+z+1}\right],
\end{multline*}
which matches the $j$--dependent factors in \eqref{eq_kernel_section_1}. The remaining $z$--integral for the $j$--th term becomes
\begin{align*}
    \frac{1}{(2\pi \ii)}\oint_{C_{z}[0,\infty)}\frac{\Gamma(z+2k+1) \Gamma(-\frac{z+2k}{2})}{\Gamma(z+1)}
    \frac{(u_{2})^{z}}{2j+z+1} dz,
\end{align*}
The last integral is evaluated as the sum of the residues at $z=0,2,4,\dots$. Using the fact that the residue of the Gamma function at a simple pole at $(-n)$, $n=1,2,\dots,$ is $\tfrac{(-1)^n}{n!}$, we get
\begin{multline*}
 \mathrm{Res}_{z=2i} \left[\frac{\Gamma(z+2k+1) \Gamma(-\frac{z+2k}{2})}{\Gamma(z+1)}
    \frac{(u_{2})^{z}}{2j+z+1}\right]=-2\frac{(-1)^{i+k}}{(i+k)!} \cdot\frac{(2i+2k)!}{(2i)!}\cdot \frac{(u_2)^{2i}}{2j+2i+1}
    \\=-2^{k+1}(-1)^{i+k} \frac{\prod_{a=1}^k(2i+2a-1)}{i!}\cdot \frac{(u_2)^{2i}}{2j+2i+1},
\end{multline*}
which matches the $i$--dependent factors in \eqref{eq_kernel_section_1}.
\end{proof}
\begin{prop}\label{Proposition_kernel_section_2}
$K^{'}(2k,u_{1};2k,u_{2})$ of Theorem \ref{thm:def1} for $u_{1},u_{2}\in \R_{\ge 0}$, $k\in\mathbb Z_{>0}$, is also equal to
\begin{equation}
\label{eq_kernel_section_2}
 e^{-(u_2)^{2}}\sum_{l=0}^{k-1}\frac{H_{2l}(u_1)H_{2l}(u_2)}{ 2^{2l-1} (2l)! \sqrt{\pi}},
\end{equation}
where  $H_i(u)$ are Hermite polynomials, as in \eqref{eq_Cor_intro}  and Theorem \ref{thm:ague}.
\end{prop}
\begin{proof}
Our task is to show that \eqref{eq_kernel_section_1} and \eqref{eq_kernel_section_2} match. The proof is induction in $k$. For $k=1$, this is a content of Remark \ref{Remark_match_small_k}. Subtracting \eqref{eq_kernel_section_1} at $k+1$ and at $k$, we get
\begin{multline}\label{eq_x6}
%2^{k+2}\sum_{i=0}^{\infty}\sum_{j=0}^{k}(-1)^{i+k+1}\frac{u_{1}^{2j}}{(2k-2j+1)! (2j)! \, \Gamma(j+\frac{1}{2}-k-1)}\cdot\frac{u_{2}^{2i} \cdot \prod_{a=1}^{k+1}(2i+2a-1) }{i!}\cdot\frac{1}{2j+2i+1}
%\\-2^{k+1}\sum_{i=0}^{\infty}\sum_{j=0}^{k-1}(-1)^{i+k}\frac{u_{1}^{2j}}{(2k-2j-1)! (2j)! \, \Gamma(j+\frac{1}{2}-k)}\cdot\frac{u_{2}^{2i} \cdot \prod_{a=1}^{k}(2i+2a-1) }{i!}\cdot\frac{1}{2j+2i+1}
%\\=
2^{k+1}\sum_{i=0}^{\infty}\sum_{j=0}^{k}(-1)^{i+k}\frac{u_{1}^{2j}}{(2k-2j+1)! (2j)! \, \Gamma(j+\frac{1}{2}-k)}\cdot\frac{u_{2}^{2i} \cdot \prod_{a=1}^{k}(2i+2a-1) }{i!}\cdot\frac{1}{2j+2i+1}
\\ \times \left[-2 (j-k-1/2)(2i+2k+1)-(2k-2j+1)(2k-2j) \right]
\\=2^{k+1}\sum_{j=0}^{k}\frac{u_{1}^{2j}}{(2k-2j)! (2j)! \, \Gamma(j+\frac{1}{2}-k)} \sum_{i=0}^{\infty} (-1)^{i+k} \frac{u_{2}^{2i} \cdot \prod_{a=1}^{k}(2i+2a-1) }{i!}
\end{multline}

On the other hand, subtracting \eqref{eq_kernel_section_2} at $k+1$ and at $k$, we get
\begin{equation}
\label{eq_x7}
 e^{-(u_2)^{2}}\frac{H_{2k}(u_1)H_{2k}(u_2)}{ 2^{2k-1} (2k)! \sqrt{\pi}},
\end{equation}
The explicit expression for the Hermite polynomials \cite[Chapter 5.5]{Szego} is
$$
H_{2k}(u_{1})=\sum_{j=0}^{k}(-1)^{k-j} 2^{2j} \frac{(2k)!}{(k-j)!(2j)!}u_{1}^{2j}
$$
In order to match with $u_1$--dependent part of \eqref{eq_x6} we write
$$
 \Gamma(j+\tfrac{1}{2}-k)= \frac{\Gamma(\tfrac12)}{(\tfrac{1}{2}-1) (\tfrac{1}{2}-2) \cdots (\tfrac{1}{2}+j-k)}=\frac{(-1)^{k-j} 2^{k-j} \sqrt{\pi}  }{1\cdot 3 \cdot 5 \cdots (2k-2j-1)}= \frac{(-1)^{k-j}  2^{2k-2j} \sqrt{\pi} (k-j)!   }{ (2k-2j)!}
$$
Hence, \eqref{eq_x6} gets transformed into
\begin{multline*}
 2^{k+1}\sum_{j=0}^{k}\frac{u_{1}^{2j}}{(2j)! (-1)^{k-j}  2^{2k-2j} \sqrt{\pi} (k-j)!  } \sum_{i=0}^{\infty} (-1)^{i+k} \frac{u_{2}^{2i} \cdot \prod_{a=1}^{k}(2i+2a-1) }{i!}\\= \frac{2^{1-k}}{(2k)!\sqrt{\pi}} H_{2k}(u_1) \sum_{i=0}^{\infty} (-1)^{i+k} \frac{u_{2}^{2i} \cdot \prod_{a=1}^{k}(2i+2a-1) }{i!}
\end{multline*}
Comparing with \eqref{eq_x7}, it remains to show that:
$$
e^{-(u_2)^{2}}\frac{H_{2k}(u_2)}{ 2^{k} }\stackrel{?}{=} \sum_{i=0}^{\infty} (-1)^{i+k} \frac{u_{2}^{2i} \cdot \prod_{a=1}^{k}(2i+2a-1) }{i!}.
$$
Or, equivalently,
\begin{equation}
\label{eq_x8}
H_{2k}(u_2)\stackrel{?}{=}e^{(u_2)^{2}} 2^k \sum_{i=0}^{\infty} (-1)^{i+k} \frac{u_{2}^{2i} \cdot \prod_{a=1}^{k}(2i+2a-1) }{i!}.
\end{equation}
The identity \eqref{eq_x8} is a corollary of the Rodrigues formula for Hermite polynomials \cite[Chapter 5.5]{Szego}
$$
H_{n}(u)=e^{u^{2}} (-1)^n (\partial_u)^n e^{-u^2}$$
Indeed, we have
\begin{multline}
 e^{u^{2}} (-1)^{2k} (\partial_u)^{2k} e^{-u^2}= e^{u^2}  (\partial_u)^{2k} \left[\sum_{i=0}^{\infty} \frac{(-u^2)^i}{i!}\right]=e^{u^2}  (\partial_u)^{2k} \left[\sum_{i=0}^{\infty} (-1)^{i+k}\frac{u^{2i+2k}}{(i+k)!}\right]\\= e^{u^2}   \left[\sum_{i=0}^{\infty} (-1)^{i+k}\frac{u^{2i}(2i+1)(2i+2)\cdots(2i+2k)}{(i+k)!}\right]\\=e^{u^2}  \left[\sum_{i=0}^{\infty} (-1)^{i+k}\frac{u^{2i}2^k(2i+1)(2i+3)\cdots(2i+2k-1)}{i!}\right].\qedhere
\end{multline}
\end{proof}

\begin{thm}\label{thm:check}
The point process  $\X^{'}$ in Theorem \ref{thm:def1} has the same distribution as the point process $\X$ --- the corner process of aGUE of Definition \ref{def:cornerprocess}.
\end{thm}
\begin{proof}
 {\bf Step 1.}  For each $k=1,2,\dots$, each of the processes $\X^{'}$ and $\X$ has $k$ particles on level $2k$ (i.e.\ with the first coordinate $2k$). Let us show the marginal distributions describing the joint law of these $k$ particles are the same. By Proposition \ref{Proposition_kernel_section_2}, for $\X^{'}$ these particles form a determinantal point process with kernel \eqref{eq_kernel_section_2}. By Theorem \ref{thm:ague}, for $\X$ these particles also form a determinantal point process with kernel given by plugging $x=y=2k$ into \eqref{eq_Cor_extended}. The kernels match and, hence, so do the point processes.

 {\bf Step 2.} Let us now fix $k$ and compare the joint distribution of particles in $\X^{'}$ and in $\X$ on levels $l=1,2,\dots,2k$. By part \ref{eq_aGUE_cond} in Proposition \ref{prop:aGUE}, the conditional law of the particles on levels $l=1,2,\dots,2k-1$ given $k$ particles on level $k$ is uniform (subject to interlacing conditions) for $\X$. On the other hand, $\X^{'}$ possesses the same conditional uniformity, because it was obtained by a scaling limit from uniformly random $\mathrm{PYT}$ $\T_{\Delta_{n}}$ and conditional uniformity is preserved in limit transitions. Combining with Step 1, we conclude that the joint distributions of particles on levels $l=1,2,\dots,2k$ are the same for $\X^{'}$ and $\X$. Since $k$ is arbitrary, we are done.
\end{proof}

\subsection{Proof of Theorem \ref{thm:syt}}
Theorems \ref{thm:def1} and \ref{thm:def2} show that the point process $\X^{'}_{n}$, $\X^{'}_{k,n}$ induced from $n^{\frac{1}{2}}(1-\T(l,m))$, the rescaled entries of PYT of shapes $\Delta_{n}$ and $\Delta_{n}\setminus (n-k,k)$, converge weakly as $n\to\infty$ to $\X^{'}$ and $\X^{'}_{k}$, respectively. In particular, treating the entries of the PYT of these two shapes as infinite random vectors, they converges in the sense of finite dimensional distribution to the positions of the corresponding particles in $\X^{'}_{n}$ and $\X^{'}_{k,n}$, respectively.

On the other hand, by Lemma \ref{lemma:coupling}, there's a coupling of PYT $\T$ and SYT $T$ that, for any $(l,m)$, $$n^{\frac{1}{2}}\left[\Bigl(1-\tfrac{T(l,m)}{N}\Bigr)-\Bigl(1-\T(l,m)\Bigr)\right]$$
converges to $0$ in probability. Combining the above two results, we obtain the $n\to\infty$ convergence for finite dimensional distributions of $$\left\{n^{\frac{1}{2}}\left(1-\frac{T(l,m)}{N}\right)\right\}_{l\ge 2,\,\, 1\le m \le\lfloor\frac{l}{2}\rfloor}$$
to the limits given by particles of  $\X^{'}$ and $\X^{'}_{k}$. It remains to use Theorem \ref{thm:check} to match $\X^{'}$ with $\X$. \qed

\section{Asymptotics for spacings}\label{sec:correspondence}

We start this section by studying general rotationally-invariant point processes on a circle, giving two different definitions of a spacing between particles in such processes, and proving that their distribution functions satisfy relations, which are discrete versions of \eqref{eq_density_of_spacing_1} and \eqref{eq_density_of_spacing_2}. When specialized to the random sorting networks setting, these spacings are precisely the ones discussed in the introduction. After that we recall the Edelman-Greene bijection between standard Young tableaux and sorting networks. In the last subsection
we combine all the ingredients to finish the proofs of Theorems \ref{thm:firstswap}, \ref{thm:gap1}, \ref{thm:derivative}, and \ref{thm:correspondence}.

\subsection{Generalities on spacings} \label{Section_spacings_general}
Consider a random point process $\mathcal P$ on a discrete circle of length $K\in \Z_{\ge 2}$. We index the possible positions of the particles by $1$, $2$, \dots, $K$ in the clockwise order and refer to the midpoint between $1$ and $K$ as the origin, see Figure \ref{Figure_circle}. Throughout this section we assume that the distribution of the point process is invariant under rotations of the circle; we also silently assume that almost surely $\mathcal P$ has at least one particle.

\begin{figure}[htpb]
    \begin{center}
        \includegraphics[scale=0.55]{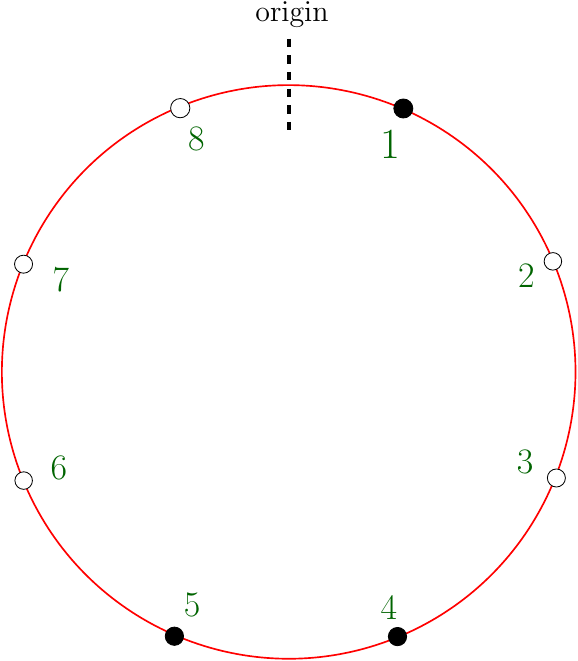} \qquad \qquad
        \includegraphics[scale=0.55]{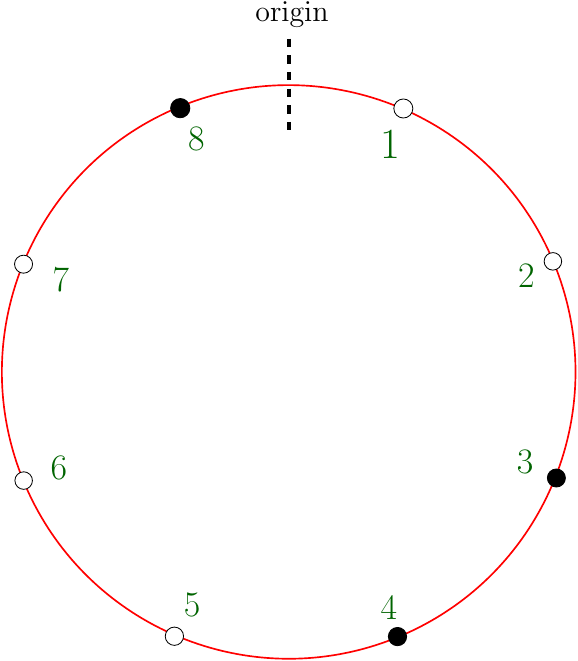}
        \caption{Two configurations of point process on discrete circle of length $K=8$ described in Example \ref{ex:circle}. Particles are shown in black.  \label{Figure_circle}}
    \end{center}
\end{figure}

We use $\mathcal P$ to define several random variables with positive integer values. We define the waiting time $W\in \{1,2,...,K\}$ to be the smallest value of $\ell\ge 1$ such that there is a particle at position $\ell$. We let the spacing $\mathrm{Sp}_{1}\in \{1,2,...,K\}$ to be the distance between two adjacent particles on the circle: one to the left from the origin and one to the right from the origin. By distance we mean here one plus the number of the holes between the particles (counted along the arc including the origin), so that if there are particles at positions $1$ and $K$, then the distance is $1$.

Next, we define the conditional spacing $\mathrm{Sp}_{2}\in \{1,2,...,K\}$ to be a random variable whose distribution is that of $W$ conditional on having a particle at position $K$. Note that $W$ and $\mathrm{Sp}_1$ are random variables (functions) on the probability space of $\mathcal P$, but $\mathrm{Sp}_2$ should be defined on a different probability space.

\begin{ex}\label{ex:circle}
Figure \ref{Figure_circle} gives two possible configurations of a point process on the discrete circle of length $K=8$. On the first configuration, $\mathrm{Sp}_{1}=4$ and $W=1$. On the second configuration, $\mathrm{Sp}_{1}=W=3$. Since there is a particle at $K$, we can also think of $\mathrm{Sp}_{2}$ being equal to $3$ for the second configuration.
\end{ex}

Let $f_1(\cdot)$, $f_2(\cdot)$, and $g(\cdot)$ be the probability mass functions of $\mathrm{Sp}_1$, $\mathrm{Sp}_2$, and $W$, respectively:
$$
 f_1(\ell)=\pr{\mathrm{Sp}_{1}=\ell}, \qquad f_2(\ell)=\pr{\mathrm{Sp}_{2}=\ell}, \qquad g(\ell)=\pr{W=\ell}, \qquad l=1,2,\dots.
$$
We also define $\rho$ to the the probability that there is a particle at position $K$ (in other words, this is the first correlation function or density of the point process $\mathcal P$).

\begin{prop}\label{prop:circle}
For any rotationally invariant point process $\mathcal P$, we have
\begin{equation}
-\Delta g(\ell)\cdot \ell=f_{1}(\ell),\qquad \ell=1,2,\dots,
\label{eq_relation_1}
\end{equation}
\begin{equation}
-\Delta g(\ell)=\rho \cdot f_{2}(\ell), \qquad \ell=1,2,\dots,
\label{eq_relation_2}
\end{equation}
where $\Delta$ is the forward difference operator: $\Delta g(\ell):=g(\ell+1)-g(\ell)$.
\end{prop}
\begin{proof} We claim that
\begin{equation}
\label{eq_x9}
-\Delta g(\ell)=\pr{\text{there are  particles at positions }K\text{ and } \ell, \text{ but not at }1,2,\dots,\ell-1}.
\end{equation}
Indeed, using rotational invariance of the law of $\mathcal P$, we have
\begin{align*}
 g(\ell)&=\pr{\text{there is a particle at position }\ell, \text{ but not at }1,2,\dots,\ell-1},\\
 g(\ell+1)&=\pr{\text{there is a particle at position }\ell, \text{ but not at }K, 1,2,\dots,\ell-1}.
\end{align*}

In order to prove \eqref{eq_relation_2},  it is now sufficient to rewrite the probability in the right-hand side of \eqref{eq_x9} as the product of probability of having a particle at $K$ (which is $\rho$) and conditional probability.

In order to prove \eqref{eq_relation_1}, note that by definition
\begin{equation}
\label{eq_x10}
 f_1(\ell)=\sum_{a=0}^{\ell-1}  \pr{\text{there are  particles at }K-a\text{ and } \ell-a, \text{ but not at }K-a+1,\dots,\ell-a-1}.
\end{equation}
By rotational invariance all terms in the right-hand side of \eqref{eq_x10} are equal. There are $\ell$ of them and each one is computed by \eqref{eq_x9}, leading to \eqref{eq_relation_1}.
\end{proof}

\begin{remark} A continuous version of Proposition \ref{prop:circle} for translationally invariant point processes on the real line $\mathbb R$ says that the following is true (under technical regularity assumptions on the point process, which we do not spell out, and which are needed to guarantee the existence of all the densities below): Suppose that $\mathrm{Sp}_{1}$, $\mathrm{Sp}_{2}$ and $W$ are spacing (distance between closest particles to the left and to the right from the origin), conditional spacing, and waiting time for an arrival of a particle, and let $f_1(x)$, $f_2(x)$, $g(x)$, $x\ge 0$ be probability densities of these random variables. Then we have
\begin{align*}
 -x \cdot \partial_x g(x)&=f_{1}(x),\\
 -\partial_x g(x)&=\rho\cdot f_{2}(x).
\end{align*}
where $\rho$ is equal to the first correlation function for the process.
\end{remark}

\subsection{Edelman-Greene bijection}
\label{Section_EG}

The relation of the random Young tableaux (which we were studying in Sections \ref{sec:pre} and \ref{sec:calculate}) with the random sorting networks relies on the Edelman--Green bijection \cite{EG}, which we now present.

The bijection takes a standard Young tableau $T$ of shape $\Delta_{n}$ as an input and outputs a sorting network. This correspondence maps the uniform measure on standard Young tableaux of shape $\Delta_n$ to the uniform measure on sorting networks of size $n$.

The bijection proceeds through the following algorithm, in which we use the standard $(i,j)$ coordinate system for Young tableaux, as in Section \ref{subsec:youngtableaux}:

\begin{enumerate}
\item Given a standard Young tableau $T$, find the box $(n-\ell,\ell)$ which contains the largest entry of $T$. Necessarily, this entry is $N=\frac{n(n-1)}{2}$ and $ 1\le \ell \le n-1$.

\item Set the first swap $s_{1}$ of the corresponding sorting network to be $\ell$.

\item Define the sliding path in the following way: the first box of the path is $(n-\ell,\ell)$. Compare the entries at $(n-\ell-1,\ell)$ and at $(n-\ell,\ell-1)$, and take the box with the larger entry as the second box (if only one of the boxes is inside the Young diagram, then take that one). Repeat this procedure (each time decreasing by $1$ either the first of the second coordinate of the box and moving in the direction of the larger entry), until you arrive at the box $(1,1)$, which is the last box of the sliding path.

\item Slide the entries on the sliding path in such a way that the maximal entry $N$ is removed and all the remaining entries on the path are moved to the neighboring boxes on the path. No entry remains at $(1,1)$ after the sliding. Then we increase all entries by $1$, and fill the box $(1,1)$ with new entry $1$.

\item After transformations of steps 1-4, $T$ becomes a new standard Young tableau of shape $\Delta_{n}$. Repeat steps 1-4 additional $(N-1)$ times to get swaps $s_{2}$,$s_{3}$,...,$s_{N}$ of the sorting network.
\end{enumerate}

Using the Edelman-Green bijection, the random variable $T_{FS,n}(k)$ of Theorem \ref{thm:firstswap} --- the first time swap $k$ occurs --- gets recast in terms of the uniformly random Young tableau.

\begin{lemma}
\label{Lemma_first_swap_through_SYT}
 Let $T_{\Delta_{n}}(l,m)$ be a uniformly random standard Young tableau of shape $\Delta_n$ in the rotated coordinate system of Section \ref{subsec:rotation}. The following distributional identity holds:
\begin{equation}
\label{eq_x12}
 T_{FS,n}(k)\stackrel{d}{=} N+1 -  T_{\Delta_{n}}(2k,1)
\end{equation}
\end{lemma}
\begin{proof}
 In each iteration of the Edelman--Greene algorithm the entry $(n-k,k)$ (in the standard $(i,j)$ coordinate system for Young tableaux, as in Section \ref{subsec:youngtableaux}) grows by $1$ until it becomes $N$, at which point the swap $s_k$ is added to the sorting network for the first time. The entry at $(n-k,k)$ in the $(i,j)$ coordinate system is the same as the entry $(2k,1)$ in the rotated $(l,m)$ coordinate system, leading to \eqref{eq_x12}.
\end{proof}

We can also compute the distribution of the conditional spacing $\widehat{\mathrm{Sp}}_{k,n}$ of Definition \ref{def:second}.

\begin{lemma}
\label{Lemma_conditional_spacing_through_SYT}
Let $T_{\Delta_{n}\setminus(n-k,k)}(l,m)$ be a uniformly random standard Young tableau of shape ${\Delta_{n}\setminus(n-k,k)}$ in the rotated coordinate system of Section \ref{subsec:rotation}. The following distributional identity holds:
\begin{equation}
\label{eq_x11}
 \widehat{\mathrm{Sp}}_{k,n}\stackrel{d}{=}
\begin{cases}
N-\max\bigl(T_{\Delta_{n}\setminus(n-k,k)} (2k-1,1), T_{\Delta_{n}\setminus(n-k,k)} (2k+1,1)\bigr), & \text{if}\;\; 2\le k\le n-2; \\
N-T_{\Delta_{n}\setminus(n-k,k)} (3,1), & \text{if}\;\; k=1; \\
N-T_{\Delta_{n}\setminus(n-k,k)} (2n-3,1), & \text{if}\;\; k=n-1.
\end{cases}
\end{equation}
\end{lemma}
\begin{proof}
 In order to obtain the law of $\widehat{\mathrm{Sp}}_{k,n}$, we need to condition on the first swap being $k$; through the Edelman-Green bijection this is the same as conditioning on the largest entry $N$ in the tableau to be in the box $(n-k,k)$ in the $(i,j)$ coordinate system. Note that if we condition a uniformly random standard Young tableau of shape $\Delta_n$ on the position of largest entry $T(n-k,k)=N$, then the rest is a uniformly random standard Young tableau of shape $\Delta_n\setminus(n-k,k)$.

 Once we do conditioning, the value of $1+\widehat{\mathrm{Sp}}_{k,n}$ becomes the number of the iterations of the Edelman-Greene algorithm when the largest entry of the tableau is at $(n-k,k)$ for the second time. After the first step of the algorithm, the entry at $(n-k,k)$ is
 $$1+\max\bigl(T_{\Delta_{n}\setminus(n-k,k)} (2k-1,1), T_{\Delta_{n}\setminus(n-k,k)} (2k+1,1)\bigr)$$
  and at each further step the entry grows by $1$, until it reaches $N$. Hence, we need $N-\max\bigl(T_{\Delta_{n}\setminus(n-k,k)} (2k-1,1), T_{\Delta_{n}\setminus(n-k,k)} (2k+1,1)\bigr)$ additional iterations, matching the first case in \eqref{eq_x11}. The second and the third cases correspond to the situations when $(n-k,k)$ has only one neighboring box in the tableau.
\end{proof}

\subsection{Proofs of the main theorems}

\begin{proof}[Proof of Theorem \ref{thm:firstswap}] By Lemma \ref{Lemma_first_swap_through_SYT}, we need to find the $n\to\infty$ asymptotics of
$$\Tb_{FS,n}(k)\stackrel{d}{=}N+1 -  T_{\Delta_{n}}(2k,1)= \frac{N}{n^{1/2}} \cdot n^{1/2}\left( \frac{1}{N}+1- \frac{T_{\Delta_{n}}(2k,1)}{N}\right).
$$
Recalling $N\sim \tfrac{1}{2}n^2$ and using \eqref{eq_limit_1} in Theorem \ref{thm:syt}, we get \eqref{eq_x13} with the right-hand side being $\tfrac{1}{2}$ times the law of the smallest eigenvalue in $2k\times 2k$ aGUE. The law of the latter is given by the Fredholm determinant through Theorem \ref{thm:ague} for the case $x=y=2k$ and general properties of the determinantal point processes (cf.\cite{Bor}).
\end{proof}
\begin{remark}
When $k=1$, there is only one particle at level $x=2$ and  $\Tb_{FS}(1)$ equals in distribution to the coordinate of this particle. The density of this random variable is given by $K(2,u;2,u)=\frac{2}{\sqrt{\pi}}e^{-u^{2}}$, $u\ge 0$, as in Remark \ref{Remark_match_small_k}. The $k=1$ result was first obtained in \cite[Theorem 1.7]{Rozinov}.
\end{remark}

\begin{proof}[Proof of Theorem \ref{thm:correspondence}] This is a combination of Theorem \ref{thm:derivative},  Lemma \ref{Lemma_conditional_spacing_through_SYT} with \eqref{eq_limit_2} in Theorem \ref{thm:syt}.
\end{proof}
\begin{remark}
When $k=1$, $Z'_{3}$ is the coordinate of the unique particle of $\X^{'}_{1}$ on level $x=3$. The density of this random variable is given by $K^{'}_{1}(3,u;3,u)=2ue^{-u^{2}}$, $u\ge 0$. Although without giving a formal definition, \cite{Rozinov} used our second definition of spacing implicitly, and obtained this $k=1$ result in \cite[Theorem 1.7]{Rozinov}. \end{remark}

\begin{proof}[Proof of Theorem \ref{thm:gap1}] Let $\Lambda_{-}=a-X\in \Z_{\ge 0}$ and $\Lambda_{+}=Y-a\in \Z_{> 0}$. We have $\Lambda_{-},\Lambda_{+}\ge 0$, and $\Lambda_{-}+\Lambda_{+}=\mathrm{Sp}_{k,n}$. Note that $\Lambda_{+}\overset{d}{=}\Tb_{FS,n}(k)$.

By translation invariance of Corollary \ref{Corollary_translation_invariance},
$$\pr{\Lambda_{-}\ge r,\, \Lambda_{+}\ge q}=\pr{\Lambda_{-}\ge r',\, \Lambda_{+}\ge q'},$$
if $r,q,r',q'\ge 0$ and $r+q=r'+q'$.  Hence, for any $u,v\ge 0$ by Theorem \ref{thm:firstswap}
\begin{align*}
\pr{2\frac{\Lambda_{-}}{n^{\frac{3}{2}}}\ge u,\, 2\frac{\Lambda_{+}}{n^{\frac{3}{2}}}\ge v}
= &\pr{\Lambda_{-}\ge  u\frac{n^{\frac{3}{2}}}{2},\, \Lambda_{+}\ge v\frac{n^{\frac{3}{2}}}{2}}
= \pr{\Lambda_{-}\ge 0,\, \Lambda_{+}\ge  u\frac{n^{\frac{3}{2}}}{2}+  v\frac{n^{\frac{3}{2}}}{2}}\\
= &\pr{2\frac{\Lambda_{+}}{n^{\frac{3}{2}}}\ge\frac{ u\frac{n^{\frac{3}{2}}}{2}+  v\frac{n^{\frac{3}{2}}}{2}}{\frac{n^{\frac{3}{2}}}{2}}} \stackrel{n\to\infty}{\longrightarrow} W(u+v),
\end{align*}
where $W(\cdot)$ is the distribution function $\pr{\Tb(k)\ge \cdot}$. Smoothness of $W(\cdot)$ and the above computation show that $\bigl(2\frac{\Lambda_{-}}{n^{\frac{3}{2}}},2\frac{\Lambda_{+}}{n^{\frac{3}{2}}}\bigr)$ converges in distribution as $n\to\infty$ to a random vector $(Z_{1},Z_{2})$, whose probability density we denote $f(x,y)$, $x,y\ge 0$. We also let $F(a,b):=\pr{Z_{1}> a, Z_{2}> b}$; we already know that $F(a,b)=W(a+b)$, $a,b\ge 0$. Denoting $F_1$ and $F_2$ the partial derivatives of $F$ in the first and second coordinates, respectively, we have $F_1(a,b)=F_2(a,b)=W'(a+b)$. Recall that $Z_{1}+Z_{2}$ is the scaled limit of $\mathrm{Sp}_{k,n}$ that we are interested in. We have:
\begin{align*}
\pr{Z_{1}+Z_{2}> b} & = \int_{0}^{b}\int_{b-y}^{\infty}f(x,y)dxdy+\int_{b}^{\infty}\int_{0}^{\infty}f(x,y)dxdy \\
 & = \int_{0}^{b}\left[-F_{2}(b-y,y)\right]dy+\int_{0}^{\infty}\left[-F_{1}(x,b)\right]dx\\
 & = -\int_{0}^{b}W'(b)dy-\int_{0}^{\infty}W'(x+b)dx\\
 & = -bW'(b)-\int_{b}^{\infty}W'(x)dx
\end{align*}
Differentiating the last identity in $b$ we get the desired \eqref{eq_density_of_spacing_1}.
\end{proof}

\begin{proof}[Proof of Theorem \ref{thm:derivative}] We fix $k$ and deal with periodic extension of the sorting network of Definition \ref{Definition_periodic_extension}. Given a random sorting network $(s_t)_{t\in\mathbb Z}$, we define a point process $\mathcal P$ on discrete circle of length $2N$ by declaring that a spot $i$ is occupied by a particle if and only if $s_i=k$; in other words, this is the point process of times of swaps $s_k$. By Corollary \ref{Corollary_translation_invariance} the process $\mathcal P$ is rotationally invariant and, therefore, the results of Section \ref{Section_spacings_general} apply. Since we would like to find the distribution of the conditional spacing, we rely on \eqref{eq_relation_2} --- the desired quantity is $f_2(\ell)$ in that formula, which is being connected to $g(\ell)$ and $\rho$. The asymptotics of $g(\ell)$ is computed in Theorem \ref{thm:firstswap}. In order to find $\rho$, we notice that by the Edelman--Greene bijection (cf.\  \cite[Proposition 9]{AHRV}), it is computed as
$$\rho=\frac{\#\{\text{Standard Young tableaux of shape } \Delta_n\setminus(n-k,k)\}}{\#\{\text{Standard Young tableaux of shape } \Delta_n\}}.$$
Both numerator and denominator are explicitly known (they can be computed by the hook length formula \cite{FRT}). Applying the Stirling's formula, we get as $n\to\infty$:
\begin{align}
\notag\frac{1}{\rho}= & \binom{n}{2}\frac{\prod_{1\le j<k}(2k-2j)}{\prod_{1\le j<k}(2k-1-2j)}\, \frac{\prod_{k< i\le n}(2i-2k)}{\prod_{k< i\le n}(2i-2k-1)}\\
= & \binom{n}{2}\frac{(2n-2-2k)!!}{(2n-1-2k)!!}\frac{(2k-2)!!}{(2k-1)!!}\sim \frac{\sqrt{\pi}}{4}n^{\frac{3}{2}}\frac{(2k-2)!!}{(2k-1)!!}.\label{eq_x14}
\end{align}
Now take  $0\le a< b$ and sum \eqref{eq_relation_2} over $\lfloor a\frac{n^{\frac{3}{2}}}{2}\rfloor\le \ell \le\lfloor b\frac{n^{\frac{3}{2}}}{2}\rfloor$. We get
\begin{equation}
\label{eq_x15}
\sum_{\lfloor a\frac{n^{\frac{3}{2}}}{2}\rfloor}^{\lfloor b\frac{n^{\frac{3}{2}}}{2}\rfloor}f_{2}(\ell)=\frac{1}{\rho}\left(g\left(\lfloor a\frac{n^{\frac{3}{2}}}{2}\rfloor\right)-g\left(\lfloor b\frac{n^{\frac{3}{2}}}{2}\rfloor+1\right) \right).
\end{equation}
We send $n\to\infty$ using Theorem \ref{thm:firstswap}, \eqref{eq_x14} and the following lemma, whose proof we leave as an exercise for the reader (the monotonicity condition is implied by \eqref{eq_relation_2}).
\begin{lemma}
 Let $\xi_1,\xi_2,\dots$ be $\mathbb Z_{\ge 0}$--valued random variables, such that for each $n$ the probabilities $\pr{\xi_n=\ell}$ depend on $\ell=0,1,2\dots$ in a monotone way. Suppose that we have a distributional limit $\lim\limits_{n\to\infty} \frac{2 \xi_n}{n^{3/2}}\stackrel{d}{=}\xi$, where $\xi$ is an absolutely continuous random variable with density $h(x)$. Then for each $a\ge 0$
 \begin{equation}
   \lim_{n\to\infty}  \frac{n^{3/2}}{2}\pr{\xi_n=\lfloor a\frac{n^{\frac{3}{2}}}{2}\rfloor}= h(a).
 \end{equation}
\end{lemma}
Hence, in terms of $\widehat{\mathrm{Sp}}_{k,n}$ of Theorem \ref{thm:derivative}, as $n\to\infty$ \eqref{eq_x15} implies
\begin{equation}
 \lim_{n\to\infty} \pr{  a\le  2\frac{\widehat{\mathrm{Sp}}_{k,n}}{n^{\frac{3}{2}}}\le b}=  \frac{\sqrt{\pi}}{2}\frac{(2k-2)!!}{(2k-1)!!}\left[\frac{\partial}{\partial b}\bigl(\pr{ \Tb_{\rm{FS}}(k)> b}\bigr)-\frac{\partial}{\partial a}\bigl(\pr{ \Tb_{\rm{FS}}(k)> a}\bigr)\right].
\end{equation}
Thus, $ 2\frac{\widehat{\mathrm{Sp}}_{k,n}}{n^{\frac{3}{2}}}$ converges in distribution to a random variable of density given by \eqref{eq_density_of_spacing_2}.
\end{proof}

\begin{remark}
We expect that an alternative proof of Theorem \ref{thm:gap1} can be obtained by following the lines of the just presented argument, but using \eqref{eq_relation_1} instead of \eqref{eq_relation_2}.
\end{remark}
\begin{remark}
 Theorem \ref{thm:derivative} and Theorem \ref{thm:correspondence} give two different expressions for the distribution of the same random variable $\widehat Z_k$. It would be interesting to find a direct (i.e.\ avoiding taking the limit from the sorting networks) proof that these expressions coincide.
\end{remark}

\end{document}